\documentclass{article}

\usepackage{fullpage}
\usepackage[utf8]{inputenc}
\usepackage[T1]{fontenc}
\usepackage{graphicx}
\usepackage{amsmath,amsfonts,amssymb,mathtools}
\usepackage{stmaryrd}
\usepackage{graphicx}
\usepackage{mathrsfs,dsfont}
\usepackage{amsmath,amsthm,amsthm}
\usepackage{hyperref}
\usepackage{enumerate}
\usepackage{color}
\usepackage{subfigure}
\usepackage{amscd}
\usepackage{epstopdf}

\newcommand{\E}{\mathbb{E}}
\newcommand{\R}{\mathbb{R}}
\newcommand{\N}{\mathbb{N}}
\newcommand*\dd{\mathop{}\!\mathrm{d}}

\newcommand{\HH}{\mathbb{H}}
\newcommand{\EE}{E}

\newcommand{\Z}{\mathcal{Z}}

\newtheorem{theo}{Theorem}[section]

\newtheorem{rem}[theo]{Remark}

\newtheorem{propo}[theo]{Proposition}
\newtheorem{lemma}[theo]{Lemma}

\newtheorem{ass}[theo]{Assumption}

\begin{document}
	
\title{Analysis of splitting schemes for stochastic evolution equations with non-Lipschitz nonlinearities driven by fractional noise
	\thanks{This work was supported by the Natural Science Foundation of Shaanxi
		Province (Grant No.2025JC-YBMS-277), Shaanxi Fundamental Science Research Project for Mathematics and Physics (Grant No.23JSY044).}
	% General acknowledgments should be placed at the end of the article.
}


\author{\begin{tabular}[t]{c@{\extracolsep{0em}}c}
		{\large  Xiao-Li Ding$^{a}$, Charles-Edouard Br\'{e}hier$^{b}$, Dehua Wang$^{c}$}\\
		{\normalsize \it $^a$Department of Mathematics, Xi'an Polytechnic  University, Shaanxi 710048, China.} \\
		{\normalsize \it $^b$Universit{e} de Pau et des Pays de l'Adour, E2S UPPA, CNRS, LMAP, Pau, France.} \\
		{\normalsize \it $^c$Department of Mathematics, Xi'an Technological University, Shaanxi 710021, China.} \\
		{\normalsize \it
			(dingding0605@126.com; charles-edouard.brehier@univ-pau.fr; dhwang@xatu.edu.cn) }
\end{tabular}}

\maketitle

\begin{abstract}
We propose a novel time-splitting scheme for a class of semilinear stochastic evolution equations driven by cylindrical fractional noise. The nonlinearity is decomposed as the sum of a one-sided, non-globally, Lipschitz continuous function, and of a globally Lipschitz continuous function. The proposed scheme is based on a splitting strategy, where the first nonlinearity is treated using the exact flow of an associated differential equation, and the second one is treated by an explicit Euler approximation. We prove mean-square, strong error estimates for the proposed scheme and show that the order of convergence is $H-1/4$, where $H\in(1/4,1)$ is the Hurst index. For the proof, we establish new regularity results for real-valued and infinite dimensional fractional Ornstein-Uhlenbeck process depending on the value of the Hurst parameter $H$. Numerical experiments illustrate the main result of this manuscript.
	%Include keywords and mathematical subject classification numbers as needed.

\end{abstract}

	
\section{Introduction}
Fractional Brownian motion (fBm) presents both randomness and strong scale-free correlations. Unlike the standard Brownian motion, fractional Brownian motion is a non-Markovian stochastic process characterized by power-law autocorrelation function, which is determined by the Hurst parameter $H\in (0,1)$. Stochastic differential equations driven by fBm constitute a fundamental class of stochastic systems, and they have been extensively applied in modeling physical phenomena exhibiting long-range dependence and anomalous behavior, such as financial time series exhibiting non-Markovian dynamics \cite{6}, complex physical processes demonstrating anomalous diffusion patterns \cite{7}, generative inpainting \cite{Lobashev-2025}, and so on. 

The analysis of well-posedness, regularity properties, and numerical approximations for stochastic differential equations with fBm are not only a fundamental topic but also a vitally important research area. In recent years these topics have witnessed a lot of attention, see for instance~\cite{Duncan-2000,Maslowski,Mishura,Neuenkirch,5,11,10,Ding-2024,Ding-2024-1}. It should be noted that, however, significant challenges remain in the study of such equations with non-Lipschitz coefficients driven by fBm \cite{Liu-2022,Zhou-2023}. One of the main reasons is that the fBm is neither a Markov process nor a semimartingale, hence the classical stochastic analysis methods can not be applied directly. In this situation, some new stochastic analysis tools need to be further explored.

Stochastic partial differential equations with non-Lipschitz coefficients driven by fractional noise are an important class of stochastic partial differential equations. The typical example is stochastic Allen-Cahn equation, which has been applied into modeling a two phase system driven by the Ginzburg-Landan energy and other rare stochastic systems \cite{Allen-1979,Kohn-2007,Lindgren-2015}. It is well-known that the main difficulty of dealing with such stochastic partial differential equation is the presence of non globally Lipschitz continuous nonlinearities. In the last decade,  progress has been made in the theoretical analysis and numerical approximations, using numerical schemes based on exponential integrator \cite{Becker-2017-1}, tamed discretization of the nonlinear coefficient \cite{Lindgren-2015-1}, Wong-Zakai approxiamtion \cite{Liu-2017-1} and full-phase flow splitting method \cite{Brehier-18-01,Brehier-19-02,Brehier-20-02}, see also \cite{Higham-2002,Zamani-2010,Becker-2018,Qi-2019}. All those references deal with systems driven by standard Wiener processes, i.e. with Hurst parameter $H=1/2$.
However, we note that there remains room for further improvement and refinement. For example, although the full-phase flow splitting method \cite{Brehier-18-01,Brehier-19-02,Brehier-20-02} is feasible for the stochastic Allen--Cahn equation, one may not be able to implement such schemes for other models. This suggests the application of semi-phase flow methods, where the drift is decomposed into two parts, one being treated using the exact flow, the other one being treated by an explicit Euler approximation. In addition, the analysis for the fractional case $H\neq 1/2$ needs to be developped.

%Or, for some stochastic differential equations with polynomial coefficient, it is not easy to obtain its analytical full-phase flow function. Indeed, one can observe that although the first part $-X^3$ of $F=-X^3+X$ is not globally Lipschitz continuous, the second linear part $X$ in $F$ is globally Lipschitz continuous. This motivates us to develop a new splitting strategy, which consists in solving separately the first part $-X^3$ of the nonlinear coefficient and the other contributions from the linear part and noise part. Compared to the existing splitting approaches, the new splitting scheme maybe demonstrate superior capability in handling the polynomial coefficient.

In this manuscript, we consider stochastic evolution equations driven by fractional noise, which may be written as
\[
\left\lbrace
\begin{aligned}
&\dd X(t)+AX(t)\dd t=\bigl(F(X(t))+G(X(t))\bigr)\dd t+\dd B^H(t),\quad \forall~t\ge 0,\\
&X(0)=X_0,
\end{aligned}
\right.
\]
where the assumptions on $A$, $F$ and $G$ will be given in Section \ref{S.2}, $\{B^H(t)\}_{t\in [0,T]}$ is a cylindrical fractional Brownian motion depending on the Hurst parameter $H$, and $X_0$ is an initial value. For well-posedness, it is assumed that $H\in(1/4,1)$. The case $H\neq 1/2$ has already been treated in previous work so it is omitted. The nonlinearity $F$ is non-globally Lipschitz continuous, but it satisfies a one-sided Lischitz continuity property, whereas the nonlinearity $G$ is assumed to be globally Lipschitz continuous.

We propose to compute approximate the solution of the stochastic evolution equation using the following numerical scheme
\[
\left\lbrace
\begin{aligned}
&X_{n+1}=S(\Delta t)\Phi_{\Delta t}(X_n)+A^{-1}\bigl(I-S(\Delta t)\bigr)G(X_n)+S(\Delta t)\Delta B_n^H,\quad \forall~n\in\{0,\ldots,N-1\},\\
&X_0=x_0,
\end{aligned}
\right.
\]
where $\Delta t$ is the time-step size, $\Delta B_n^H$ are increments of the cylindrical fractional Brownian motion, and $\Phi_{\Delta t}$ denotes the phase-flow associated with the nonlinearity $F$.

The main result of this manuscript is Theorem~\ref{theo:main}: one obtains strong error estimates, with order of convergence $H-1/4$, with respect to the time-step size $\Delta t$. This result is illustrated by numerical experiments in Section~\ref{S.5}. 

In order to prove Theorem~\ref{theo:main}, we provide auxiliary results, mainly on the regularity properties of the fractional Ornstein--Uhlenbeck and its numerical approximation, obtained when $F=G=0$. Different arguments are used when $H\in(1/2,1)$ or when $H\in(1/4,1/2)$. For completeness, detailed proofs are provided in the appendices.

This article is organized as follows. The setting and assumptions are presented in Section~\ref{S.2}. Proposition~\ref{propo:errorXtildeXexact} recalls well-posedness and regularity properties of solutions to the stochastic evolution equation. Section~\ref{S.4} is devoted to the presentation of the numerical scheme and to the analysis of its convergence. It contains the statement of Theorem~\ref{theo:main} and the statement and proofs of several auxiliary results. Section~\ref{S.5} provides numerical experiments which illustrate the main result of this work. Section~\ref{S.6} presents some conclusion and perspectives for future works. Appendices~\ref{App.A},~\ref{app:+},~\ref{app:-} and~\ref{app:D} provide key auxiliary results and their proofs on real-valued fractional Ornstein--Uhlenbeck processes, which play a crucial role in the analysis.

\section{Setting}\label{S.2}

\subsection{Notation}

Let $\HH$ denote the separable Hilbert space $L^2((0,1);\mathbb R)$ of square-integrable real-valued functions defined on $(0,1)$. The associated inner product $\langle \cdot,\cdot \rangle$ and norm $\|\cdot\|$ are defined by  
\begin{align*}
\langle x_1, x_2\rangle=\int_0^1x_1(\xi)x_2(\xi)\dd\xi, \quad \|x\|=\Bigg(\int_0^1x(\xi)^2\dd\xi\Bigg)^{1/2},\quad \forall~x_1,x_2,x\in\HH.
\end{align*}
Moreover, let $E$ denote the Banach space ${\mathcal C}([0,1];\R)$ of continuous functions from $[0,1]$ to $\R$. The associated norm $\|\cdot\|_\EE$ is defined by
\begin{align*}
\|x\|_\EE=\underset{\xi\in[0,1]}\max~|x(\xi)|. 
\end{align*}

Finally, let $\mathcal{L}(\HH)$ denote the space of bounded linear operators from $\HH$ to $\HH$, and let $\mathcal{L}_2(\HH)$ denote the space of Hilbert--Schmidt linear operators from $\HH$ to $\HH$. The associated norm on $\mathcal{L}(\HH)$, resp. on $\mathcal{L}_2(\HH)$, is denoted by $\|\cdot\|_{\mathcal{L}(\HH)}$, resp. by $\|\cdot\|_{\mathcal{L}_2(\HH)}$.

\subsection{Linear operator}

In this work, the linear operator $-A$ is the realization of the Laplace operator with homogeneous Dirichlet boundary conditions on the interval $(0,1)$.
\begin{ass}\label{ass:A}
The linear operator $A:\HH\to\HH$ is the unbounded linear operator on $\HH$ with domain
\[
D(A)=H_0^1(0,1);\R) \cap H^2(0,1;\R),
\]
such that for all $x\in D(A)$, one has
\[
Ax=-x''.
\]
\end{ass}

For all $k\in\N$, set $\lambda_k=k^2\pi^2$ and $e_k=\sqrt{2}\sin\bigl(k\pi \cdot\bigr)$. Then $\bigl(e_k\bigr)_{k\in\N}$ is a complete orthonormal system  of the Hilbert space $\HH$, and for all $k\in\N$ one has $Ae_k=\lambda_ke_k$.

For all $\alpha\in\R^+$, the linear operator $A^\alpha$ is defined by the domain
\[
D(A^\alpha)=\{x\in\HH;~\sum_{k=1}^{\infty}\lambda_k^{2\alpha}\langle x,e_k\rangle^2<\infty\},
\]
and such that for all $x\in D(A^\alpha)$ one has
\[
A^\alpha x=\sum_{k=1}^{\infty}\lambda_k^\alpha \langle x,e_k \rangle e_k.
\]
In the sequel the notation
\[
\HH^\alpha=D(A^{\frac{\alpha}{2}}),\quad \forall~\alpha\in\R^+,
\]
is employed. The norm $\|\cdot\|_{\HH^\alpha}$ is defined by
\[
\|x\|_{\HH^\alpha}^2=\|A^{\frac{\alpha}{2}}x\|^2=\sum_{k\in\N}\lambda_k^{\alpha}\langle x,e_k\rangle^2,\qquad \forall~x\in\HH^\alpha.
\]

Moreover, for all $\alpha\in\R^+$, the linear operator $A^{-\alpha}$ is the bounded linear operator on $\HH$ defined by
\[
A^{-\alpha}x=\sum_{k=1}^{\infty}\lambda_k^{-\alpha} \langle x,e_k \rangle e_k,\quad \forall~x\in\HH.
\]

The linear operator $A$ generates a semigroup $\bigl(S(t)\bigr)_{t\ge 0}$, where for all $t\ge 0$ the linear operator $S(t)$ is defined by
\[
S(t)x=\sum_{k=1}^\infty e^{-\lambda_kt}\langle x,e_k\rangle e_k,\quad \forall~x\in\HH.
\]

Proposition~\ref{propo:semigroup} states standard properties of the semigroup $\bigl(S(t)\bigr)_{t\ge 0}$. Please see for instance~\cite{Yosida,Pruss}.

\begin{propo}\label{propo:semigroup}
The semigroup $\bigl(S(t)\bigr)_{t\ge 0}$ satisfies the following properties.

$\bullet$ For all $t\geq 0$, $S(t)$ is a bounded linear operator on $\HH$, and one has $\|S(t)\|_{{\mathcal L}(\HH)}\leq 1$.

$\bullet$ For all $t\geq 0$, $S(t)$ is a bounded linear operator on $\EE$, and one has $\|S(t)\|_{{\mathcal L}(\EE)}\leq 1$.

$\bullet$ For all $\alpha\ge 0$, there exists $C_\alpha\in(0,\infty)$ such that one has
\[
\|A^\alpha S(t)\|_{{\mathcal L}({\mathbb H})}\leq Ct^{-\alpha},\quad \forall~t>0.
\]
Moreover, for all $\rho\in [0,1]$, there exists $C_\rho\in(0,\infty)$ such that one has
\[
\|A^{-\rho}(I-S(t))\|_{{\mathcal L}({\mathbb H})}\leq Ct^\rho,\quad \forall~t\ge 0.  
\]

$\bullet$ For all $\rho\in [0,1]$, there exists $C_\rho\in(0,\infty)$ such that for all $t_2\ge t_1>0$ one has
\begin{align}\label{eq.11}
\|S(t_2)-S(t_1)\|_{{\mathcal L}({\mathbb H})}\leq C(t_2-t_1)^\rho t_1^{-\rho}.  
\end{align}
\end{propo}

It is straightforward to check that the eigenfunctions $\{e_k;k\in\N\}$ satisfy the following properties (\cite{Prato}): for all $k\in\N$, one has $e_k\in\EE$, and there exists $C\in(0,\infty)$ such that for all $k\in\N$ one has
\[
\underset{\xi\in [0,1]}\max~|e_k(\xi)|\leq C, \quad \underset{\xi\in [0,1]}\max~|\nabla e_k(\xi)|\leq C\lambda_k^{1/2}.
\]
As a result, for all $\alpha\in[0,1]$, there exists $C_\alpha\in(0,\infty)$ such that one has
\begin{equation}\label{eq:Holder_ek}
|e_k(\xi_2)-e_k(\xi_1)|\le C_\alpha \lambda_k^{\frac{\alpha}{2}}|\xi_2-\xi_1|^\alpha,\quad \forall~\xi_1,\xi_2\in[0,1].
\end{equation}
Moreover, one has the following property: for all $\alpha>1/2$, one has $\HH^\alpha\subset \EE$, and there exists $C_\alpha\in(0,\infty)$ such that
\begin{equation}\label{eq:Sobolevembedding}
\|x\|_\EE\le C_\alpha \|x\|_{\HH^\alpha},\quad \forall~x\in\HH^\alpha.
\end{equation}

\subsection{Nonlinearities}

The first nonlinearity $F:\EE\to \EE$ is defined depending on a real-valued continuous mapping $f:\R\to\R$, such that one has
\[
F(x)(\xi)=f\bigl(x(\xi)\bigr),\qquad \forall~\xi\in[0,1],\quad \forall~x\in\EE.
\]
The mapping $f$ is not assumed to be globally Lipschitz continuous, instead it is assumed to be one-sided Lipschitz continuous.
\begin{ass}\label{ass:f}
There exists ${\rm L}_f\in[0,\infty)$ such that one has
\[
\bigl(z_2-z_1\bigr)\bigl(f(z_2)-f(z_1)\bigr)\le {\rm L}_f\bigl(z_2-z_1\bigr)^2,\quad \forall~z_1,z_2\in\R.
\]
Moreover, there exists $q\in\N$ such that one has
\[
|f(z_2)-f(z_1)|\le {\rm L}_f\bigl(1+|z_1|^{2q}+|z_2|^{2q}\bigr)|z_2-z_1|,\quad \forall~z_1,z_2\in\R.
\]
\end{ass}

For the construction of the scheme, let us introduce the flow $\bigl(\phi_{s}\bigr)_{s\ge 0}$ associated with the nonlinearity $f$: for any $z\in\R$, the mapping $s\ge 0\mapsto \phi_s(z)=z_s$ is the solution to the ordinary differential equation
\[
\dot{z}_s=f(z_s),\quad s\ge 0;\qquad z_0=z.
\]
Under Assumption~\ref{ass:f}, the flow is well-defined. Let us define
\[
\psi_s(z)=\frac{\phi_s(z)-z}{s},\qquad \forall~s>0, \forall~z\in\R,
\]
and denote also $\psi_0=f$.

\begin{lemma}\label{lem:flow}
There exist $q'\in\N$ and ${\rm C}\in(0,\infty)$ such that one has
\begin{align}
&\underset{s\in[0,1]}\sup~|\phi_{s}(z_2)-\phi_{s}(z_1)|\le e^{{\rm C}s}|z_2-z_1|,\quad \forall~z_1,z_2\in\R,\\
&\underset{s\in[0,1]}\sup~\bigl(z_2-z_1\bigr)\bigl(\psi_s(z_2)-\psi_s(z_1)\bigr)\le {\rm C}\bigl(z_2-z_1\bigr)^2,\quad \forall~z_1,z_2\in\R,\\
&\underset{s\in[0,1]}\sup~|\psi_{s}(z_2)-\psi_{s}(z_1)|\le {\rm C}\bigl(1+|z_1|^{2q'}+|z_2|^{2q'}\bigr),\quad \forall~z_1,z_2\in\R,\\
&|\psi_s(z)-\psi_0(z)|\le {\rm C}s(1+|z|^{2q'+1}),\quad \forall~s\ge 0, \forall~z\in\R.
\end{align}
\end{lemma}

For the implementation of the scheme, the flow needs to be known. This is satisfied for instance for polynomial nonlinearities $f(z)=-z^{2q+1}$, for some integer $q\in\N$: in that case one has
\[
\Phi_s(z)=\frac{z}{\bigl(1+2qz^{2q}s\bigr)^{\frac{1}{2q}}},\qquad \forall~s\ge 0, \forall~z\in\R.
\]

For any time $s\ge 0$, the nonlinearities $\Phi_s:E\to E$ and $\Psi_s:E\to E$ are defined by
\[
\Phi_s(x)(\xi)=\phi_{s}\bigl(x(\xi)\bigr),\quad \Psi_s(x)(\xi)=\psi_{s}\bigl(x(\xi)\bigr),\qquad \forall~\xi\in[0,1], \forall~x\in E.
\]

The second linearity $G$ is defined similarly: there exists a real-valued continuous mapping $g:\R\to\R$, such that one has
\[
G(x)(\xi)=g\bigl(x(\xi)\bigr),\qquad \forall~\xi\in[0,1],\forall~x\in\EE.
\]
Contrary to the mapping $f$, it is assumed that the mapping $g$ is globally Lipschitz continuous.
\begin{ass}\label{ass:g}
There exists ${\rm L}_g\in[0,\infty)$ such that one has
\[
\big|g(z_2)-g(z_1)\big|\le {\rm L}_g|z_2-z_1|,\quad \forall~z_1,z_2\in\R.
\]
\end{ass}
As a result, the nonlinearity can also be seen as a mapping $G:\HH\to\HH$, which is globally Lipschitz continuous.

For instance, one may choose $g(z)=z$ or $g(z)=\sin(z)$.

\subsection{Fractional Brownian motion}

Let $H\in(0,1)$ denote the so-called Hurst index. Consider a sequence $\bigl(\beta_k^H\bigr)_{k\in\N}$ of independent real-valued fractional Brownian motions with Hurst index $H$ defined on a complete probability space $\bigl(\Omega,\mathcal{F},\mathbb{P}\bigr)$.

Let us recall that for each $k\in\N$, the process $\bigl(\beta_k^H(t)\bigr)_{t\ge 0}$ is a Gaussian process characterized by its mean $t\ge 0\mapsto \E[\beta_k^H(t)]$ and its covariance function $t,s\ge 0\mapsto R^H(t,s)=\E[\beta_k^H(t)\beta_k^H(s)]$, which are given by
\begin{align*}
&\E[\beta_k^H(t)]=0,\quad \forall~t\ge 0,\\
&\E[\beta_k^H(t)\beta_k^H(s)]=\frac{1}{2}\bigl(t^{2H}+s^{2H}-|t-s|^{2H}\bigr),\quad \forall~t,s\ge 0.
\end{align*}
In particular, one has $\beta_k^H(0)=0$.

Let $\bigl(\beta_k\bigr)_{k\in\N}$ be a sequence of standard real-valued Brownian motions, which correspond to choosing the Hurst index $H=1/2$. When $H\neq 1/2$, for each $k\in\N$ the fractional Brownian motion $\beta_k^H$ may be defined as
\[
\beta_k^H(t)=\int_{0}^{t}\mathcal{K}_H(t,s)\dd \beta_k(s),\quad \forall~t\ge 0,
\]
where $\mathcal{K}_H$ is a kernel, defined differently whether $H\in(1/2,1)$ or $H\in(0,1/2)$. Precisely, for all $H\in(0,1)\setminus\{1/2\}$, set
\[
c_H
=
\left\lbrace
\begin{aligned}
&\Bigl(\frac{H\bigl(2H-1\bigr)}{B\bigl(2-2H,H-\frac12\bigr)}\Bigr)^{1/2},\qquad H\in(\frac12,1),\\
&-\Bigl(\frac{H\bigl(1-2H\bigr)}{2B\bigl(1-2H,H+\frac12\bigr)}\Bigr)^{1/2},\qquad H\in(0,\frac12),
\end{aligned}
\right.
\]
then for all $t>s>0$ one has
\[
K_H(t,s)
=
\left\lbrace
\begin{aligned}
&c_H\int_{s}^{t}\frac{\tau^{H-\frac12}}{s^{H-\frac12}}(\tau-s)^{H-\frac{3}{2}}\dd \tau,\qquad H\in(\frac12,1);\\
&c_H\Bigl(\frac{t^{H-\frac12}}{s^{H-\frac12}}\frac{(t-s)^{H-\frac12}}{H-\frac12}-s^{\frac12-H}\int_{s}^{t}(\tau-s)^{H-\frac12}\tau^{H-\frac32}\dd \tau\Bigr),\qquad H\in(0,\frac12).
\end{aligned}
\right.
\]
Observe that for all values of $H\in(0,1)\setminus\{1/2\}$ one has
\[
\frac{\partial K_H(t,s)}{\partial t}=c_H\frac{t^{H-\frac12}}{s^{H-\frac12}}(t-s)^{H-\frac32},\quad \forall~t>s>0.
\]

Let $T\in(0,\infty)$, then the stochastic integral
\[
\int_{0}^{T}\phi(t)\dd\beta_k^H(t)
\]
may be defined in different ways. In this manuscript, we only need with deterministic integrands $\phi:[0,T]\to\R$. Let us first assume that $\phi$ is of class $\mathcal{C}^1$. We define
\[
\int_0^T\phi(t)\dd\beta_k^H(t)=\phi(T)\beta_k^H(T)-\int_{0}^{T}\phi'(t)\beta_k^H(t)\dd t.
\]
Given the representation of the fractional Brownian motion $\beta_k^H$ in terms of a standard Brownian motion $\beta_k$ depending on the kernel $K_H$, one has the expression
\[
\int_{0}^{T}\phi(t)\dd\beta_k^H(t)=\int_{0}^{T}\mathcal{K}_{H,T}^\star\phi (t)\dd \beta_k(t),
\]
where the linear operator $\mathcal{K}_{H,T}^\star$ is given by
\[
\mathcal{K}_{H,T}^\star\phi (t)=\phi(T)K_H(T,t)-\int_{t}^{T}\phi'(s)K_H(s,t)\dd s,\quad \forall~t\in[0,T].
\]
It is important to define the stochastic integral for integrands $\phi$ which are not of class $\mathcal{C}^1$. This is made possible using alternative expression for $\mathcal{K}_{H,T}^\star\phi$. Using an integration by parts arguments, one obtains for all $t\in[0,T]$
\[
\mathcal{K}_{H,T}^\star\phi(t)
=
\left\lbrace
\begin{aligned}
&\int_{t}^{T}\phi(s)\frac{\partial K_H(s,t)}{\partial t}\dd s,\qquad H\in(\frac12,1);\\
&\phi(t)K_H(T,t)+\int_{t}^{T}\bigl[\phi(s)-\phi(t)\bigr]\frac{\partial K_H(s,t)}{\partial t}\dd s,\qquad H\in(0,\frac12).
\end{aligned}
\right.
\]
Applying the It\^o isometry property for the It\^o stochastic integral with respect to standard Brownian motion, one obtains the isometry property
\begin{equation}\label{eq:Ito}
\E[\big|\int_{0}^{T}\phi(t)\dd\beta_k^H(t)\big|^2]=\int_{0}^{T}\big|\mathcal{K}_{H,T}^\star\phi (t)\big|^2 \dd t.
\end{equation}
It is straightforward to generalize the definition of the stochastic integral to integrands $\phi$ such that the right-hand side of the above isometry property is finite.

When $H\in(\frac12,1)$, one obtains the following expression for the second-order moment:
\begin{equation}\label{eq:Ito+}
\int_{0}^{T}\big|\mathcal{K}_{H,T}^\star\phi (t)\big|^2 \dd t=H(2H-1)\int_{0}^{T}\int_{0}^{T}\phi(s_1)\phi(s_2)|s_2-s_1|^{2H-2}\dd s_1 \dd s_2.
\end{equation}
When $H\in(0,\frac12)$, there is no similar simple expressions. Instead, the following alternative expression for $\mathcal{K}_{H,T}^\star\phi (t)$ is employed: one has
\begin{equation}\label{eq:KH-}
\mathcal{K}_{H,T}^\star\phi (t)=c_H\Bigl(\frac{\phi(t)}{(H-\frac12)(T-t)^{\frac12-H}}-t^{\frac12-H}\int_{t}^{T}\frac{t^{H-\frac12}\phi(t)-s^{H-\frac12}\phi(s)}{(s-t)^{\frac32-H}}\dd s\Bigr).
\end{equation}

To conclude this section, let us introduce the cylindrical fractional Brownian motion $\bigl(B^H(t)\bigr)_{t\ge 0}$ defined by
\[
B^H(t)=\sum_{k\in\N}\beta_k^H(t)e_k.
\]
The above series does not converge in $\HH$, however it is possible to show that $A^{-\alpha}B^H(t)$ converges if $\alpha>1/4$: indeed, one has
\[
\E[\|A^{-\alpha}B^H(t)\|^2]=t^{2H}\sum_{k\in\N}\lambda_k^{-2\alpha}<\infty,\qquad \forall~\alpha>1/4, \forall~t\ge 0.
\]
More generally, if $\theta\in\mathcal{L}_2(\HH)$ is an Hilbert--Schmidt linear operator on $\HH$, then the random variable $\theta B^H(t)$ is a well-defined $\HH$-valued Gaussian random variable for all $t\ge 0$, with
\[
\E[\|\theta B^H(t)\|^2]=t^{2H}\|\theta\|_{\mathcal{L}_2(\HH)}^2.
\]
For an integrand $\Theta:[0,T]\to\mathcal{L}_2(\HH)$, one defines the $\HH$-valued stochastic integral with respect to the cylindrical fractional Brownian motion, as
\[
\int_{0}^{T}\Theta(t)\dd B^H(t)=\sum_{k,\ell\in\N}\int_{0}^{T}\langle\Theta(t)e_\ell,e_k\rangle \dd \beta_\ell^H(t) e_k.
\]

More details about fBm can be refered to \cite{4,Hu-2005,Nualart-2006}. 
\section{Preliminary properties}\label{S.3}

\subsection{Fractional Ornstein--Uhlenbeck process}

Let us introduce the process $\bigl(Z^H(t)\bigr)_{t\ge 0}$ defined by
\begin{equation}\label{eq:defZH}
Z^H(t)=\int_{0}^{t}S(t-s)\dd B^H(s),\quad \forall~t\ge 0,
\end{equation}
which is considered as the mild solution to the stochastic evolution equation
\begin{equation}\label{eq:SPDEZH}
\left\lbrace
\begin{aligned}
&\dd Z^H(t)+AZ^H(t)\dd t=\dd B^H(t),\\
&Z^H(0)=0.
\end{aligned}
\right.
\end{equation}
Below we show that $\bigl(Z^H(t)\bigr)_{t\ge 0}$ is well-defined with values in $\HH$ if and only if $H>1/4$. Moreover, we provide moment bounds in $\HH^\alpha$ for appropriate values of $\alpha$ and in $\EE$, and temporal regularity results.

Note one can express $Z^H(t)$ as a combination of real-valued Gaussian random variables $\Z_k^H(t)$: for all $t\ge 0$, one has
\begin{equation}\label{eq:decompZH}
Z^H(t)=\sum_{k\in\N}\Z_k^H(t)e_k,
\end{equation}
where for all $k\in\N$ the real-valued fractional Ornstein--Uhlenbeck process $\bigl(\Z_k^H(t)\bigr)_{t\ge 0}$ is defined by
\[
\Z_k^H(t)=\sum_{\ell\in\N}\int_{0}^{T}\langle S(t-s)e_\ell,e_k\rangle \dd \beta_\ell^H(s)=\int_{0}^{T}e^{-\lambda_k(t-s)}\dd\beta_k^H(s).
\]
Therefore the analysis proceeds by proving moment bounds and temporal regularity properties for the real-valued fractional Ornstein--Uhlenbeck processes $\bigl(\Z_k^H(t)\bigr)_{t\ge 0}$. One needs to pay attention to the dependence with respect to $\lambda_k$. We refer to Appendices~\ref{app:+} and~\ref{app:-} for results on real-valued fractional Ornstein--Uhlenbeck processes when $H\in(1/2,1)$ and $H\in(0,1/2)$ respectively.

Verifying that the process $\bigl(Z^H(t)\bigr)_{t\ge 0}$ takes values in the Hilbert space $\HH$ requires to assume that $H>1/4$.
\begin{propo}\label{propo:ZH}
For all $H\in(1/4,1/2)\cup(1/2,1)$, for all $t\ge 0$, almost surely $Z^H(t)\in\HH$, and for all $p\in\N$ one has
\begin{equation}\label{eq:momentsZH}
\underset{t\ge 0}\sup~\E[\|Z^H(t)\|^{2p}]<\infty.
\end{equation}
Moreover, for all $\alpha\in(0,2H-1/2)$, for all $t\ge 0$, almost surely $Z^H(t)\in\HH^\alpha$, and for all $p\in\N$ one has
\begin{equation}\label{eq:momentsZHalpha}
\underset{t\ge 0}\sup~\E[\|Z^H(t)\|_{\HH^\alpha}^{2p}]<\infty.
\end{equation}
\end{propo}

\begin{proof}[Proof of Proposition~\ref{propo:ZH}]
Since the considered random variables are Gaussian, it suffices to deal with the case $p=1$.

Owing to the decomposition~\eqref{eq:decompZH} of $Z^H(t)$, for all $t\ge 0$ one has
\[
\E[\|Z^H(t)\|^2]=\sum_{k\in\N}\E[|\mathcal{Z}_k^{H}(t)|^2].
\]
Applying the inequality~\eqref{eq:moment1DH+} from Lemma~\ref{lem:moment1DH+} when $H\in(1/2,1)$ or the inequality~\eqref{eq:moment1DH-} from Lemma~\ref{lem:moment1DH-} when $H\in(1/4,1/2)$, for all $k\in\N$ one has
\[
\underset{t\ge 0}\sup~\E[|\mathcal{Z}_k^{H}(t)|^2]\le C_H\lambda_k^{-2H}.
\]
To obtain the moment bounds~\eqref{eq:momentsZH}, it suffices to observe that one has
\[
\sum_{k\in\N}\lambda_k^{-2H}<\infty\quad\Longleftrightarrow\quad H>1/4.
\]
Next, for all $\alpha\ge 0$ and all $t\ge 0$, one has
\[
\E[\|Z^H(t)\|_{\HH^\alpha}^{2p}]=\E[\|A^{\frac{\alpha}{2}}Z^H(t)\|^2]=\sum_{k\in\N}\lambda_k^{\alpha}\E[|\mathcal{Z}_k^{H}(t)|^2]
\le C_H\sum_{k\in\N}\lambda_k^{\alpha-2H},
\]
and one has
\[
\sum_{k\in\N}\lambda_k^{\alpha-2H}<\infty\quad\Longleftrightarrow\quad 2H-\alpha>1/2.
\]
This concludes the proof of Proposition~\ref{propo:ZH}.
\end{proof}

Next, let us provide moment bounds of $Z^H(t)$ in the Banach space $\EE$. When $H\in(1/2,1)$, this is straightforward consequence of Proposition~\ref{propo:ZH} due to the embedding property~\eqref{eq:Sobolevembedding}. When $H\in(1/4,1/2)$, one needs to use a version of the Kolmogorov regularity criterion.

\begin{propo}\label{propo:ZH_E}
For all $H\in(1/4,1/2)\cup(1/2,1)$, for all $t\ge 0$, almost surely $Z^H(t)\in\EE$. Moreover, for all $p\in\N$ one has
\begin{equation}\label{eq:momentsZH_E}
\underset{t\ge 0}\sup~\E[\|Z^H(t)\|_\EE^{2p}]<\infty.
\end{equation}
\end{propo}

\begin{proof}[Proof of Proposition~\ref{propo:ZH_E}]
First, assume that $H\in(1/2,1)$. In that case, one has $2H-1/2>1/2$, therefore one may choose $\alpha\in(1/2,2H-1/2)$, and combining the embedding property~\eqref{eq:Sobolevembedding} and the moment bounds~\eqref{eq:momentsZH} from Proposition~\ref{propo:ZH}, one obtains
\[
\underset{t\ge 0}\sup~\E[\|Z^H(t)\|_\EE^{2p}]\le C_{\alpha}^{2p}\underset{t\ge 0}\sup~\E[\|Z^H(t)\|_{\HH^\alpha}^{2p}]<\infty.
\]

Second, assume that $H\in(1/4,1/2)$. For all $t\ge 0$ and all $\xi\in[0,1]$, one has
\[
Z^H(t,\xi)=\sum_{k\in\N}\mathcal{Z}_k^H(t)e_k(\xi).
\]
Observe that one has $Z^H(t,0)=Z^H(t,1)=0$ for all $t\ge 0$.

Let $\alpha\in(0,2H-1/2)$, then using the inequality~\eqref{eq:Holder_ek}, for all $\xi_1,\xi_2\in[0,1]$, one has for all $t\ge 0$
\begin{align*}
\E[\|Z^H(t,\xi_2)-Z^H(t,\xi_1)\|^2]&=\sum_{k\in\N}\E[|\mathcal{Z}_k^H(t)|^2]\big|e_k(\xi_2)-e_k(\xi_1)\big|^2\\
&\le C_\alpha \sum_{k\in\N}\lambda_k^{\alpha-2H}|\xi_2-\xi_1|^{2\alpha}\\
&\le C_{H,\alpha}|\xi_2-\xi_1|^{2\alpha},
\end{align*}
with $C_{H,\alpha}\in(0,\infty)$.

Since the random variable $Z^H(t,\xi_2)-Z^H(t,\xi_1)$ is Gaussian, for all $p\in\N$ there exists $C_{p,H,\alpha}\in(0,\infty)$ such that
\[
\underset{t\ge 0}\sup~\E[\|Z^H(t,\xi_2)-Z^H(t,\xi_1)\|^{2p}]\le C_{p,H,\alpha}|\xi_2-\xi_1|^{2p\alpha}.
\]
Choosing a sufficiently large $p_\alpha\in\N$, such that $2p_\alpha\alpha>1$, one may apply the Kolmogorov regularity criterion: for all $p\in\N$ such that $p\ge p_\alpha$, there exists $C_{p,H,\alpha}\in(0,\infty)$ such that
\[
\underset{t\ge 0}\sup~\E\bigl[\underset{\xi_1,\xi_2\in[0,1], \xi_1\neq \xi_2}\sup~\frac{\|Z^H(t,\xi_2)-Z^H(t,\xi_1)\|^{2p}}{|\xi_2-\xi_1|^{2p\alpha-1}}\bigr]<\infty.
\]
In particular, for all $t\ge 0$ one has $Z^H(t)\in\EE$. Choosing $\xi_1=0$ and $\xi_2=\xi$ in the above, and recalling that $Z^H(t,0)=0$ for all $t\ge 0$, for all $p\ge p_\alpha$ one obtains
\[
\underset{t\ge 0}\sup~\E\bigl[\underset{\xi\in[0,1]}\sup~\|Z^H(t,\xi)\|^{2p}]<\infty.
\]
Therefore one obtains the moment bounds~\eqref{eq:momentsZH_E} and the proof of Proposition~\ref{propo:ZH_E} is completed.
\end{proof}

Finally, let us provide temporal regularity properties for the process $\bigl(Z^H(t)\bigr)_{t\ge 0}$.
\begin{propo}\label{propo:ZHregul}
For all $p\in\N$, $H\in(1/4,1/2)\cup(1/2,1)$ and $\alpha\in(0,2H-1/2)$, there exists $C_{p,H,\alpha}\in(0,\infty)$ such that one has
\begin{equation}\label{eq:regulZH}
\E[\|Z^H(t_2)-Z^H(t_1)\|^{2p}]\le C_{p,H,\alpha}|t_2-t_1|^{p\alpha},\qquad \forall~t_2\ge t_1\ge 0.
\end{equation}
\end{propo}

\begin{proof}[Proof of Proposition~\ref{propo:ZHregul}]
Like in the proof of Proposition~\ref{propo:ZH}, it suffices to consider the case $p=1$ since the random variables $Z^H(t_2)-Z^H(t_1)$ are Gaussian.

Let $\alpha\in(0,2H-\frac12)$. Recall the decomposition~\ref{eq:decompZH} of $Z^H(t)$. Applying the inequality~\eqref{eq:temp1DH+} from Lemma~\ref{lem:temp1DH+} when $H\in(1/2,1)$ or the inequality~\eqref{eq:temp1DH-} from Lemma~\ref{lem:temp1DH-} when $H\in(1/4,1/2)$, one obtains
\begin{align*}
\E[\|Z^H(t_2)-Z^H(t_1)\|^2]&=\sum_{k\in\N}\E[|\Z_k^H(t_2)-\Z_k^H(t_1)|^2]\\
&\le C_{\alpha,H}\sum_{k\in\N}\lambda_k^{\alpha-2H}|t_2-t_1|^\alpha.
\end{align*}
Note that
\[
\sum_{k\in\N}\lambda_k^{\alpha-2H}<\infty.
\]
As a result, one obtains the inequality~\eqref{eq:regulZH} and the proof of Proposition~\ref{propo:ZHregul} is completed.
\end{proof}

\subsection{Well-posedness of the SPDE}

Let $T\in(0,\infty)$ and $x_0\in\HH$. The process $\bigl(X(t)\bigr)_{t\in[0,T]}$ is a mild solution to the stochastic evolution equation
\begin{equation}\label{eq:SPDE}
\left\lbrace
\begin{aligned}
&\dd X(t)+AX(t)\dd t=\bigl(F(X(t))+G(X(t))\bigr)\dd t+\dd B^H(t),\\
&X(0)=x_0,
\end{aligned}
\right.
\end{equation}
if one has for all $t\in[0,T]$
\begin{equation}\label{eq:mild}
X(t)=S(t)x_0+\int_{0}^{t}S(t-s)\bigl(F(X(s))+G(X(s))\bigr)\dd s+\int_{0}^{t}S(t-s)\dd B(s).
\end{equation}
Letting
\[
Y(t)=X(t)-Z^H(t),\quad \forall~t\ge 0,
\]
it is equivalent to consider the evolution equation
\[
\dd Y(t)+AY(t)\dd t=\bigl(F(Y(t)+Z^H(t))+G(Y(t)+Z^H(t))\bigr)\dd t.
\]
By standard techniques, one obtains the following result.

\begin{propo}\label{propo:exact}
Let $H\in(1/4,1/2)\cup(1/2,1)$.

For all $T\in(0,\infty)$ and $x_0\in\EE$, the stochastic evolution equation~\eqref{eq:SPDE} admits a unique mild solution (in the sense of~\eqref{eq:mild}), such that $X(t)\in\EE$ for all $t\in[0,T]$.

Moreover, for all $p\in\N$, $T\in(0,\infty)$ and $x_0\in\EE$, there exists $C_{p,H}(T,\|x_0\|_\EE)\in(0,\infty)$ such that
\begin{equation}\label{eq:momentsexact}
\underset{t\in[0,T]}\sup~\bigl(\E[\|X(t)\|_{\EE}^p]\bigr)^{\frac{1}{p}}\le C_{p,H}(T,\|x_0\|_E).
\end{equation}

Finally, for all $p\in\N$, $\alpha\in(0,2H-1/2)$, $T\in(0,\infty)$ and $x_0\in\HH^\alpha \cap \EE$, there exists a positive real number $C_{p,H,\alpha}(T,\|x_0\|_{\EE},\|x_0\|_{\HH^\alpha})\in(0,\infty)$ such that one has
\begin{equation}\label{eq:incrementsexact}
\bigl(\E[\|X(t_2)-X(t_1)\|^{p}]\bigr)^{\frac{1}{p}}\le C_{p,H,\alpha}(T,\|x_0\|_{\EE},\|x_0\|_{\HH^\alpha})|t_2-t_1|^{\frac{\alpha}{2}},\qquad \forall~t_1,t_2\in[0,T].
\end{equation}
\end{propo}

\section{Numerical scheme}\label{S.4}

Let $T\in(0,\infty)$ be an arbitrary time, and define the time-step size $\Delta t=T/N$, for some $M\in\N$.

For all $n\in\{0,\ldots,N\}$, set $t_n=n\Delta t$, and for all $n\in\{0,\ldots,N\}$ define the increments of the cylindrical fractional Brownian motion
\[
\Delta B_n^H=B^H(t_{n+1})-B^H(t_n).
\]

We consider the following numerical scheme: set
\begin{equation}\label{eq:scheme}
\left\lbrace
\begin{aligned}
&X_{n+1}=S(\Delta t)\Phi_{\Delta t}(X_n)+A^{-1}\bigl(I-S(\Delta t)\bigr)G(X_n)+S(\Delta t)\Delta B_n^H,\quad \forall~n\in\{0,\ldots,N-1\},\\
&X_0=x_0.
\end{aligned}
\right.
\end{equation}
Noting that $\Phi_{\Delta t}=Id+\Delta t\Psi_{\Delta t}$, the scheme can equivalently be written as
\[
X_{n+1}=X_n+\Delta tS(\Delta t)\Psi_{\Delta t}(X_n)+A^{-1}\bigl(I-S(\Delta t)\bigr)G(X_n)+S(\Delta t)\Delta B_n^H.
\]

The discrete-time process $\bigl(X_n\bigr)_{0\le n\le N}$ can be described using a discrete-time mild formulation: for all $n\in\{0,\ldots,N\}$, one has
\begin{align*}
X_n&=S(\Delta t)^nx_0+\Delta t\sum_{j=0}^{n-1}S(\Delta t)^{n-j}\Psi_{\Delta t}(X_j)+\sum_{j=0}^{n-1}A^{-1}\bigl(I-S(\Delta t)\bigr)S(\Delta t)^{n-j-1}G(X_j)+\sum_{j=0}^{n-1}S(\Delta t)^{n-j}\Delta B_j^H\\
&=S(t_n)x_0+\Delta t\sum_{j=0}^{n-1}S(t_n-t_j)\Psi_{\Delta t}(X_j)+\sum_{j=0}^{n-1}A^{-1}\bigl(I-S(\Delta t)\bigr)S(t_n-t_{j+1})G(X_j)+\sum_{j=0}^{n-1}S(t_n-t_j)\Delta B_j^H.
\end{align*}

The objective of this section is to prove the main result of this manuscript.
\begin{theo}\label{theo:main}
For all $p\in\N$, $\alpha\in(0,2H-1/2)$, $T\in(0,\infty)$ and $x_0\in\HH^\alpha\cap\EE$, there exists a positive real number $C_{p,H,\alpha}(T,\|x_0\|_{\EE},\|x_0\|_{\HH^\alpha})\in(0,\infty)$ such that for all $\Delta t=T/N$ with $N\in\N$, one has
\begin{equation}\label{eq:errorscheme}
\underset{n=0,\ldots,N}\sup~\bigl(\E[\|X_n-X(t_n)\|_{\HH}^{p}]\bigr)^{\frac{1}{p}}\le C_{p,H,\alpha}(T,\|x_0\|_{\EE},\|x_0\|_{\HH^\alpha})\Delta t^{\frac{\alpha}{2}}.
\end{equation}
\end{theo}

The proof of Theorem~\ref{theo:main} follows from several auxiliary results.

First, define the auxiliary discrete-time process $\bigl(Z_n^H\bigr)_{n\ge 0}$ defined by
\begin{equation}
Z_n^H=\sum_{j=0}^{n-1}S(\Delta t)^{n-j}\Delta B_j^H=\sum_{j=0}^{n-1}S(t_n-t_j)\Delta B_j^H,\quad \forall~n\ge 0,
\end{equation}
which is the solution to the scheme when the nonlinearities $F$ and $G$ are removed:
\[
\left\lbrace
\begin{aligned}
&Z_{n+1}^H=S(\Delta t)Z_n^{H}+S(\Delta t)\Delta B_n^H,\quad \forall~n\ge 0,\\
&Z_0^H=0.
\end{aligned}
\right.
\]

Note that one can express $Z_n^H$ as a combination of real-valued Gaussian random variables $\Z_{k,n}^H$: one has
\begin{equation}\label{eq:decompZnH}
Z_n^H=\sum_{k\in\N}\Z_{k,n}^{H}e_k,
\end{equation}
where for all $k\in\N$ the real-valued discrete-time process $\bigl(\Z_{k,n}^{H}\bigr)_{n\ge 0}$ is defined by
\[
\Z_{k,n}^{H}=\sum_{\ell\in\N}\sum_{j=0}^{n-1}\langle S(t_n-t_j)e_\ell,e_k\rangle \Delta B_j^H=\sum_{j=0}^{n-1}e^{-\lambda_k(t_n-t_j)}\Delta \beta_{k,j}^{H},
\]
with increments of the real-valued fractional Brownian motions denoted by $\Delta \beta_{k,j}^{H}=\beta_k^{H}(t_{j+1})-\beta_k^H(t_j)$.

First, Lemmas~\ref{lem:momentnum1DH+} and~\ref{lem:momentnum1DH-} show that the discrete-time process $\bigl(Z_n^H\bigr)_{n\ge 0}$ takes values in $\HH$ and in $\EE$ when it is assumed that $H>1/4$.

\begin{propo}\label{propo:boundZnH}
For all $H\in(1/4,1/2)\cup(1/2,1)$, for all $\Delta t\in(0,1)$ and all $n\ge 0$, almost surely $Z_n^H\in \HH$ and $Z_n\in\EE$. Moreover, for all $p\in\N$, one has
\begin{equation}\label{eq:boundZnH}
\underset{\Delta t\in(0,1)}\sup~\underset{n\in\N}\sup~\bigl(\E[\|Z_n^H\|^p]\bigr)^{\frac{1}{p}}+\underset{\Delta t\in(0,1)}\sup~\underset{n\in\N}\sup~\bigl(\E[\|Z_n^H\|_{\EE}^p]\bigr)^{\frac{1}{p}}<\infty.
\end{equation}
\end{propo}

\begin{proof}[Proof of Proposition~\ref{propo:boundZnH}]
Since the considered random variables are Gaussian, it suffices to deal with the case $p=1$.

Owing to the decomposition~\eqref{eq:decompZnH}, for all $n\ge 0$ one has
\[
\E[\|Z_n^H\|^2]=\sum_{k\in\N}\E[|\Z_{k,n}^{H}|^2].
\]
Applying the inequality~\eqref{eq:momentnum1DH+} from Lemma~\ref{lem:momentnum1DH+} when $H\in(1/2,1)$ or the inequality~\eqref{eq:momentnum1DH-} from Lemma~\ref{lem:momentnum1DH-} when $H\in(1/4,1/2)$, for all $k\in\N$ one has
\[
\underset{\Delta t\in(0,1)}\sup~\underset{n\in\N}\sup~\E[|\Z_{k,n}^{H}|^2]\le C_H\lambda_k^{-2H}.
\]
When $H>1/4$, one has $\sum_{k\in\N}\lambda_k^{-2H}<\infty$, therefore one obtains moment bounds in $\HH$:
\[
\underset{\Delta t\in(0,1)}\sup~\underset{n\in\N}\sup~\E[\|Z_n^H\|^2]<\infty.
\]
Similarly, for all $\alpha\in(0,2H-\frac12)$, one has
\[
\underset{\Delta t\in(0,1)}\sup~\underset{n\in\N}\sup~\E[\|Z_n^H\|_{\HH^\alpha}^2]=\underset{\Delta t\in(0,1)}\sup~\underset{n\in\N}\sup~\E[\|(-A)^{\frac{\alpha}{2}}Z_n^H\|^2]\le \sum_{k\in\N}\lambda_k^\alpha \underset{\Delta t\in(0,1)}\sup~\underset{n\in\N}\sup~\E[|\Z_{k,n}^{H}|^2]\le C_J\sum_{k\in\N}\lambda_k^{\alpha-2H}<\infty.
\]
To prove moment bounds in $\EE$, like in the proof of Proposition~\ref{propo:ZH_E}, the two cases $H\in(1/4,1/2)$ and $H\in(1/2,1)$ are treated separately.

First, assume that $H\in(1/2,1)$. In that case, one has $2H-1/2>1/2$, therefore one may choose $\alpha\in(1/2,2H-1/2)$, and combining the embedding property~\eqref{eq:Sobolevembedding} and the moment bounds above of $Z_n^H$ in the $\HH^\alpha$ norm, one obtains
\[
\underset{\Delta t\in(0,1)}\sup~\underset{n\in\N}\sup~\E[\|Z_n^H\|_{\EE}^{2p}]\le C_\alpha^{2p}\underset{\Delta t\in(0,1)}\sup~\underset{n\in\N}\sup~\E[\|Z_n^H\|_{\HH^\alpha}^{2p}]<\infty.
\]
Second, assume that $H\in(1/4,1/2)$. For all $n\ge 0$ and all $\xi\in[0,1]$, one has
\[
Z_n^H(\xi)=\sum_{k\in\N}\Z_{k,n}^{H}e_k(\xi).
\]
Observe that one has $Z_n^{H}(0)=Z_n^H(1)=0$ for all $n\ge 0$.

Let $\alpha\in(0,2H-1/2)$, then using the inequality~\eqref{eq:Holder_ek}, for all $\xi_1,\xi_2\in[0,1]$, one has for all $n\ge 0$
\begin{align*}
\E[\|Z_n^H(\xi_2)-Z_n^H(\xi_1)\|^2]&=\sum_{k\in\N}\E[|\mathcal{Z}_{k,n}^H|^2]\big|e_k(\xi_2)-e_k(\xi_1)\big|^2\\
&\le C_\alpha \sum_{k\in\N}\lambda_k^{\alpha-2H}|\xi_2-\xi_1|^{2\alpha}\\
&\le C_{H,\alpha}|\xi_2-\xi_1|^{2\alpha},
\end{align*}
with $C_{H,\alpha}\in(0,\infty)$.

Since the random variable $Z_n^H(\xi_2)-Z_n^H(\xi_1)$ is Gaussian, for all $p\in\N$ there exists $C_{p,H,\alpha}\in(0,\infty)$ such that
\[
\underset{\Delta t\in(0,1)}\sup~\underset{n\in\N}\sup~\E[\|Z_n^H(\xi_2)-Z_n^H(\xi_1)\|^{2p}]\le C_{p,H,\alpha}|\xi_2-\xi_1|^{2p\alpha}.
\]
Choosing a sufficiently large $p_\alpha\in\N$, such that $2p_\alpha\alpha>1$, one may apply the Kolmogorov regularity criterion: for all $p\in\N$ such that $p\ge p_\alpha$, there exists $C_{p,H,\alpha}\in(0,\infty)$ such that
\[
\underset{\Delta t\in(0,1)}\sup~\underset{n\in\N}\sup~\E\bigl[\underset{\xi_1,\xi_2\in[0,1], \xi_1\neq \xi_2}\sup~\frac{\|Z_n^H(\xi_2)-Z_n^H(\xi_1)\|^{2p}}{|\xi_2-\xi_1|^{2p\alpha-1}}\bigr]<\infty.
\]
In particular, for all $n\ge 0$ one has $Z_n^H\in\EE$. Choosing $\xi_1=0$ and $\xi_2=\xi$ in the above, and recalling that $Z_n^H(0)=0$ for all $n\ge 0$, for all $p\ge p_\alpha$ one obtains
\[
\underset{\Delta t\in(0,1)}\sup~\underset{n\in\N}\sup~\E\bigl[\|Z_n^H\|_{\EE}^{2p}]=\underset{\Delta t\in(0,1)}\sup~\underset{n\in\N}\sup~\E\bigl[\underset{\xi\in[0,1]}\sup~\|Z_n^H(\xi)\|^{2p}]<\infty.
\]
The proof of Proposition~\ref{propo:boundZnH} is thus completed.
\end{proof}

\begin{propo}\label{propo:errorZ}
For all $p\in\N$ and $\alpha\in(0,2H-1/2)$, there exists $C_{p,H,\alpha}\in(0,\infty)$ such that for all $\Delta t=T/N$ with $N\in\N$, one has
\begin{equation}\label{eq:errorZ}
\underset{n\in\N}\sup~\bigl(\E[\|Z^H(t_n)-Z_n^H\|_{\HH}^{p}]\bigr)^{\frac{1}{p}}\le C_{p,H,\alpha}\Delta t^{\frac{\alpha}{2}}.
\end{equation}
\end{propo}

\begin{proof}[Proof of Proposition~\ref{propo:errorZ}]
Owing to the decompositions~\eqref{eq:decompZH} and~\eqref{eq:decompZnH} of $Z^H(t_n)$ and of $Z_n^H$ respectively, for all $n\ge 0$ one has
\[
Z^H(t_n)-Z_n^H=\sum_{k\in\N}\bigl(\Z_k^H(t_n)-\Z_{k,n}^H\bigr)e_k.
\]
It suffices to consider the case $p=1$ since the $\HH$-valued random variable $Z^H(t_n)-Z_n^H$ is Gaussian. One obtains
\[
\E[\|Z^H(t_n)-Z_n^H\|^2]=\sum_{k\in\N}\E[|\Z_k^H(t_n)-\Z_{k,n}^H|^2].
\]
Let $\alpha\in(0,2H-1/2)$. Applying the inequality~\eqref{eq:error1DH+} from Lemma~\ref{lem:error1DH+} when $H\in(1/2,1)$ or the inequality~\eqref{eq:error1DH-} from Lemma~\ref{lem:error1DH-} when $H\in(1/4,1/2)$, one has
\begin{align*}
\E[\|Z^H(t_n)-Z_n^H\|^2]\le C_{\alpha,H}\sum_{k\in\N}\lambda_k^{\alpha-2H}\Delta t^{\alpha}.
\end{align*}
Since one has $\sum_{k\in\N}\lambda_k^{\alpha-2H}<\infty$ under the condition $\alpha\in(0,2H-1/2)$, one obtains the inequality~\eqref{eq:errorZ} and the proof of Proposition~\ref{propo:errorZ} is thus completed.
\end{proof}

Proposition~\ref{propo:momentbounds} provides moment bounds in $\EE$ on the numerical solution $X_n$.

\begin{propo}\label{propo:momentbounds}
For all $p\in\N$, $\alpha\in(0,2H-1/2)$, $T\in(0,\infty)$ and $x_0\in\EE$, there exists a positive real number $C_{p,H,\alpha}(T,\|x_0\|_{\EE})\in(0,\infty)$ such that, for all $\Delta t=T/N$ with $N\in\N$, one has
\begin{equation}\label{eq:momentbounds}
\underset{n=0,\ldots,N}\sup~\bigl(\E[\|X_n\|_{\EE}^p]\bigr)^{\frac{1}{p}}\le C_{p,H}(T,\|x_0\|_\EE).
\end{equation}
\end{propo}

\begin{proof}[Proof of Proposition~\ref{propo:momentbounds}]
Introduce the auxiliary process $\bigl(Y_n\bigr)_{0\le n\le N}$ defined by
\[
Y_n=X_n-Z_n^H,\qquad \forall~n\in\{0,\ldots,N\}.
\]
Recall that $\Phi_{\Delta t}-I=\Delta t\Psi_{\Delta t}$. Then, for all $n\in\{0,\ldots,N-1\}$, one has
\begin{align*}
Y_{n+1}&=X_{n+1}-Z_{n+1}^H\\
&=S(\Delta t)\bigl[\Phi_{\Delta t}(X_n)-Z_n^H\bigr]+A^{-1}\bigl(I-S(\Delta t)\bigr)G(X_n)\\
&=S(\Delta t)\bigl[\Phi_{\Delta t}(Y_n+Z_n^H)-\Phi_{\Delta t}(Z_n^H)\bigr]+A^{-1}\bigl(I-S(\Delta t)\bigr)\bigl[G(Y_n+Z_n^H)-G(Z_n^H)\bigr]\\
&+\Delta tS(\Delta t)\Psi_{\Delta t}(Z_n^H)+A^{-1}\bigl(I-S(\Delta t)\bigr)G(Z_n^H).
\end{align*}
Observe that
\[
A^{-1}\bigl(I-S(\Delta t)\bigr)=\int_{0}^{\Delta t}S(s)\dd s,
\]
therefore
\begin{equation}\label{eq:boundAS}
\|A^{-1}\bigl(I-S(\Delta t)\bigr)\|_{\mathcal{L}(\EE)}\le \int_{0}^{\Delta t}\|S(s)\|_{\mathcal{L}(\EE)}\dd s\le \Delta t.
\end{equation}
Owing to the global Lipschitz continuity property of $\Phi_{\Delta t}$ and to the local Lipschitz continuity property of $\Psi_{\Delta t}$ stated in Lemma~\ref{lem:flow}, and recalling that $G$ is globally Lipschitz continuous, one has the following: there exists $C\in(0,\infty)$ such that one has
\[
\|Y_{n+1}\|_{\EE}\le \bigl[e^{C\Delta t}+C\Delta t\bigr]\|Y_n\|_{\EE}+C\Delta t\bigl(1+\|Z_n^H\|_{\EE}^{2q'+1}\bigr).
\]
Note that one has $e^{C\Delta t}+C\Delta t\le e^{2C\Delta t}$. For all $p\in\N$, owing to the Minkowski inequality, and applying the moment bounds~\eqref{eq:momentsexact} in $\EE$ for $Z_n^H$ from Proposition~\ref{propo:boundZnH}, there exists $C_{p,H}\in(0,\infty)$ such that for all $n\in\{0,\ldots,N-1\}$ one has
\[
\bigl(\E[\|Y_{n+1}\|_{\EE}^{p}]\bigr)^{\frac{1}{p}}\le e^{2C\Delta t}\bigl(\E[\|Y_{n}\|_{\EE}^{p}]\bigr)^{\frac{1}{p}}+C_{p,H}\Delta t.
\]
Since $Y_0=x_0$, it is straightforward to obtain the following inequality: for all $n\in\{0,\ldots,N-1\}$ one has
\[
\bigl(\E[\|Y_{n}\|_{\EE}^{p}]\bigr)^{\frac{1}{p}}\le e^{2Cn\Delta t}\|x_0\|_{\EE}+C_{p,H}n\Delta te^{2CN\Delta t}\le e^{2CT}\|x_0\|_{\EE}+C_{p,H}e^{2CT}T.
\]
This yields the inequality~\eqref{eq:momentbounds} and the proof of Proposition~\ref{propo:momentbounds} is completed.
\end{proof}

To perform the error analysis, let us introduce the auxiliary discrete-time process $\bigl(\tilde{X}_n\bigr)_{0\le n\le N}$ defined by
\begin{equation}\label{eq:tildeXn}
\tilde{X}_n=S(t_n)x_0+\Delta t\sum_{j=0}^{n-1}S(t_n-t_j)F(X(t_j))+\sum_{j=0}^{n-1}A^{-1}\bigl(I-S(\Delta t)\bigr)S(t_n-t_{j+1})G(X(t_j))+Z_n^H.
\end{equation}
Note that for all $n\in\{0,\ldots,N-1\}$, one has
\[
\tilde{X}_{n+1}=S(\Delta t)\tilde{X}_n+\Delta tS(\Delta t)F(X(t_n))+A^{-1}\bigl(I-S(\Delta t)\bigr)G(X(t_n))+S(\Delta t)\Delta B_n^H.
\]

First, one obtains moment bounds in $\EE$ for $\tilde{X}_n$.
\begin{propo}\label{propo:boundXtilde}
For all $p\in\N$, $\alpha\in(0,2H-1/2)$, $T\in(0,\infty)$ and $x_0\in\EE$, there exists a positive real number $C_{p,H,\alpha}(T,\|x_0\|_{\EE})\in(0,\infty)$ such that, for all $\Delta t=T/N$ with $N\in\N$, one has
\begin{equation}\label{eq:boundXtilde}
\underset{n=0,\ldots,N}\sup~\bigl(\E[\|\tilde{X}_n\|_{\EE}^p]\bigr)^{\frac{1}{p}}\le C_{p,H}(T,\|x_0\|_\EE).
\end{equation}
\end{propo}

\begin{proof}[Proof of Proposition~\ref{propo:boundXtilde}]
The mapping $f$ has at most polynomial growth owing to Assumption~\ref{ass:f}, and the mapping $g$ is globally Lipschitz continuous and thus has at most linear growth. For all $p\in\N$, applying the Minkowski inequality and using the upper bound~\eqref{eq:boundAS}, for all $n\in\{0,\ldots,N\}$, one has
\[
\bigl(\E[\|\tilde{X}_n\|_{\EE}^p]\bigr)^{\frac{1}{p}}\le \|x_0\|_{\EE}+C\Delta t\sum_{j=0}^{n-1}\Bigl(1+\bigl(\E[\|X(t_j)\|_{\EE}^{2qp}]\bigr)^{\frac{1}{p}}\Bigr)+C\Delta t\sum_{j=0}^{n-1}\Bigl(1+\bigl(\E[\|X(t_j)\|_{\EE}^{p}]\bigr)^{\frac{1}{p}}\Bigr)+\bigl(\E[\|Z_n^{H}\|_{\EE}^{p}]\bigr)^{\frac{1}{p}}.
\]
Applying the moment bounds~\eqref{eq:momentsexact} in $\EE$ for the exact solution $X(t)$ stated in Proposition~\ref{propo:exact} and for the Gaussian numerical solution $Z_n^H$ stated in Proposition~\ref{propo:boundZnH}, one obtains the inequality~\eqref{eq:boundXtilde} and the proof of Proposition~\ref{propo:boundXtilde} is completed.
\end{proof}

In order to prove Theorem~\ref{theo:main}, it then remains to state and prove Propositions~\ref{propo:errorXtildeXexact} and~\ref{propo:errorXtildeXnum}, which provide error estimates on $\tilde{X}_n-X(t_n)$ and $X_n-\tilde{X}_n$ respectively.

\begin{propo}\label{propo:errorXtildeXexact}
For all $p\in\N$, $\alpha\in(0,2H-1/2)$, $T\in(0,\infty)$ and $x_0\in\HH^\alpha \cap \EE$, there exists a positive real number $C_{p,H,\alpha}(T,\|x_0\|_{\EE},\|x_0\|_{\HH^\alpha})\in(0,\infty)$ such that, for all $\Delta t=T/N$ with $N\in\N$, one has
\begin{equation}\label{eq:errorXtildeXexact}
\underset{n=0,\ldots,N}\sup~\bigl(\E[\|\tilde{X}_n-X(t_n)\|_{\HH}^{p}]\bigr)^{\frac{1}{p}}\le C_{p,H,\alpha}(T,\|x_0\|_{\EE},\|x_0\|_{\HH^\alpha})\Delta t^{\frac{\alpha}{2}}.
\end{equation}
\end{propo}

\begin{proof}[Proof of Proposition~\ref{propo:errorXtildeXexact}]
Observe that for all $j\in\{0,\ldots,N-1\}$, one has
\[
A^{-1}\bigl(I-S(\Delta t)\bigr)S(t_n-t_{j+1})=\int_{t_j}^{t_{j+1}}S(t_{j+1}-s)S(t_n-t_{j+1})\dd s=\int_{t_j}^{t_{j+1}}S(t_n-s)\dd s.
\]
As a result, for all $n\in\{0,\ldots,N\}$, one has the following expression for $\tilde{X}_n$:
\[
\tilde{X}_n=S(t_n)x_0+\sum_{j=0}^{n-1}\int_{t_j}^{t_{j+1}}S(t_n-t_j)F(X(t_j))\dd s+\sum_{j=0}^{n-1}\int_{t_j}^{t_{j+1}}S(t_n-s)G(X(t_j))\dd s+Z_n^H.
\]
Moreover, for all $n\in\{0,\ldots,N\}$, one has the following expression for $X(t_n)$:
\[
X(t_n)=S(t_n)x_0+\sum_{j=0}^{n-1}\int_{t_j}^{t_{j+1}}S(t_n-s)F(X(s))\dd s+\sum_{j=0}^{n-1}\int_{t_j}^{t_{j+1}}S(t_n-s)G(X(s))\dd s+Z^H(t_n).
\]
Therefore, the error $\tilde{X}_n-X(t_n)$ can be decomposed as follows: one has
\begin{equation}
\tilde{X}_n-X(t_n)=e_n^{1}+e_n^{2}+e_n^{3}+e_n^{4},
\end{equation}
where the error terms $e_n^{1}$, $e_n^{2}$, $e_n^{3}$ and $e_n^{4}$ are defined as
\begin{align*}
e_n^{1}&=\sum_{j=0}^{n-1}\int_{t_j}^{t_{j+1}}\bigl[S(t_n-t_j)-S(t_n-s)\bigr]F(X(t_j))\dd s\\
e_n^{2}&=\sum_{j=0}^{n-1}\int_{t_j}^{t_{j+1}}S(t_n-s)\bigl[F(X(t_j))-F(X(s))\bigr]\dd s\\
e_n^{3}&=\sum_{j=0}^{n-1}\int_{t_j}^{t_{j+1}}S(t_n-s)\bigl[G(X(t_j))-G(X(s))\bigr]\dd s\\
e_n^{4}&=Z_n^H-Z^H(t_n).
\end{align*}
To deal with the error term $e_n^1$, recall that owing to Assumption~\ref{ass:f} the mapping $f$ has at most polynomial growth. Let $\alpha\in(0,2H-1/2)$. Owing to Proposition~\ref{propo:semigroup}, note that for all $j\in\{0,\ldots,n-1\}$ and $s\in[t_j,t_{j+1})$, one has
\begin{align*}
\|\bigl[S(t_n-t_j)-S(t_n-s)\bigr]F(X(t_j))\|_{\HH}&\le C_\alpha\frac{(s-t_j)^{\frac{\alpha}{2}}}{(t_n-s)^{\frac{\alpha}{2}}}\|F(X(t_j))\|_{\HH}\\
&\le C_\alpha\frac{\Delta t^{\frac{\alpha}{2}}}{(t_n-s)^{\frac{\alpha}{2}}}\Bigl(1+\|X(t_j)\|_{\EE}^{2q}\Bigr).
\end{align*}
For all $p\in\N$, applying the Minkowski inequality and the moment bounds~\eqref{eq:momentsexact} in $\EE$ for the exact solution from Proposition~\ref{propo:exact}, one obtains for all $n\in\{0,\ldots,N\}$
\begin{align*}
\bigl(\E[\|e_n^1\|_{\HH}^{p}]\bigr)^{\frac{1}{p}}&\le C_{\alpha}\Delta t^{\frac{\alpha}{2}}\underset{0\le j\le n}\sup~\Bigl(1+\bigl(\E[\|X(t_j)\|_{\EE}^{2qp}]\bigr)^{\frac{1}{p}}\Bigr)\sum_{j=0}^{n-1}\int_{t_j}^{t_{j+1}}\frac{1}{(t_n-s)^{\frac{\alpha}{2}}}\dd s\\
&\le C_{p,\alpha}(T,\|x_0\|_{\EE})\Delta t^{\frac{\alpha}{2}}.
\end{align*}
To deal with the error term $e_n^2$, one employs the local Lipschitz continuity property of $f$ from Assumption~\ref{ass:f}: for all $j\in\{0,\ldots,N-1\}$ and all $s\in[t_j,t_{j+1}]$ one has
\[
\|S(t_n-s)\bigl[F(X(t_j))-F(X(s))\bigr]\|_{\HH}\le C\bigl(1+\|X(t_j)\|_{\EE}^{2q}+\|X(s)\|_{\EE}^{2q}\bigr)\|X(s)-X(t_j)\|_{\HH}.
\]
Let $p\in\N$. Applying the Minkowski and the Cauchy--Schwarz inequalities, one has
\begin{align*}
\bigl(\E[\|e_n^2\|_{\HH}^{p}]\bigr)^{\frac{1}{p}}&\le C\sum_{j=0}^{n-1}\int_{t_j}^{t_{j+1}}\bigl(\E\bigl[\bigl(1+\|X(t_j)\|_{\EE}^{2qp}+\|X(s)\|_{\EE}^{2qp}\bigr)\|X(s)-X(t_j)\|_{\HH}^{p}\bigr]\bigr)^{\frac{1}{p}}\dd s\\
&\le C\sum_{j=0}^{n-1}\int_{t_j}^{t_{j+1}}\bigl(\E\bigl[\bigl(1+\|X(t_j)\|_{\EE}^{4qp}+\|X(s)\|_{\EE}^{4qp}\bigr)\bigr]\bigr)^{\frac{1}{2p}}\bigl(\E[\|X(s)-X(t_j)\|_{\HH}^{2p}\bigr]\bigr)^{\frac{1}{2p}}\dd s.
\end{align*}
Applying the moment bounds~\eqref{eq:momentsexact} in $\EE$ on the exact solution and the inequality~\eqref{eq:incrementsexact} from Proposition~\ref{propo:exact}, one obtains the upper bound
\[
\bigl(\E[\|e_n^2\|_{\HH}^{p}]\bigr)^{\frac{1}{p}}\le C_{p,H,\alpha}(T,\|x_0\|_{\EE},\|x_0\|_{\HH^\alpha})\Delta t^{\frac{\alpha}{2}}.
\]
To deal with the error term $e_n^3$, recall that $G$ is globally Lipschitz continuous owing to Assumption~\ref{ass:g}. Let $p\in\N$. Applying the Minkowski inequality and the inequality~\eqref{eq:incrementsexact} from Proposition~\ref{propo:exact}, one obtains the upper bounds
\begin{align*}
\bigl(\E[\|e_n^3\|_{\HH}^{p}]\bigr)^{\frac{1}{p}}&\le C\sum_{j=0}^{n-1}\int_{t_j}^{t_{j+1}}\bigl(\E[\|X(s)-X(t_j)\|_{\HH}^{p}\bigr]\bigr)^{\frac{1}{p}}\dd s\\
&\le C_{p,H,\alpha}(T,\|x_0\|_{\EE},\|x_0\|_{\HH^\alpha})\Delta t^{\frac{\alpha}{2}}
\end{align*}
Finally, for the error term $e_n^4$, it suffices to apply the error estimate~\eqref{eq:errorZ} from Proposition~\ref{propo:errorZ}: one obtains
\[
\bigl(\E[\|e_n^4\|_{\HH}^{p}]\bigr)^{\frac{1}{p}}\le C_\alpha\Delta t^{\frac{\alpha}{2}}.
\]
Gathering the estimates on the error terms $e_n^{1}$, $e_n^{2}$, $e_n^{3}$ and $e_n^{4}$ then provides the error estimate~\eqref{eq:errorXtildeXexact} and the proof of Proposition~\ref{propo:errorXtildeXexact} is completed.
\end{proof}

\begin{propo}\label{propo:errorXtildeXnum}
For all $p\in\N$, $\alpha\in(0,2H-1/2)$, $T\in(0,\infty)$ and $x_0\in\HH^\alpha \cap \EE$, there exists a positive real number $C_{p,H,\alpha}(T,\|x_0\|_{\EE},\|x_0\|_{\HH^\alpha})\in(0,\infty)$ such that, for all $\Delta t=T/N$ with $N\in\N$, one has
\begin{equation}\label{eq:errorXtildeXnum}
\underset{n=0,\ldots,N}\sup~\bigl(\E[\|X_n-\tilde{X}_n\|_{\HH}^{p}]\bigr)^{\frac{1}{p}}\le C_{p,H,\alpha}(T,\|x_0\|_{\EE},\|x_0\|_{\HH^\alpha})\Delta t^{\frac{\alpha}{2}}.
\end{equation}
\end{propo}

\begin{proof}[Proof of Proposition~\ref{propo:errorXtildeXnum}]
Recall that $F=\Psi_0$. For all $n\in\{0,\ldots,N-1\}$, one has
\[
X_{n+1}-\tilde{X}_{n+1}=S(\Delta t)\bigl(X_n-\tilde{X}_n\bigr)+\Delta tS(\Delta t)\bigl(\Psi_{\Delta t}(X_n)-\Psi_{0}(X_n)\bigr)+A^{-1}\bigl(I-S(\Delta t)\bigr)\bigl(G(X_n)-G(X(t_n)\bigr).
\]
Observe that one has
\begin{align*}
\Psi_{\Delta t}(X_n)-\Psi_{0}(X_n)&=\Psi_{\Delta t}(X_n)-\Psi_{\Delta t}(\tilde{X}_n)+\Psi_{\Delta t}(\tilde{X}_n)-\Psi_{\Delta t}(X(t_n))+\Psi_{\Delta t}(X(t_n))-\Psi_{0}(X(t_n))\\
G(X_n)-G(X(t_n)&=G(X_n)-G(\tilde{X}_n)+G(\tilde{X}_n)-G(X(t_n)).
\end{align*}
Recalling that $I+\Delta t\Psi_{\Delta t}=\Phi_{\Delta t}$, one thus obtains
\begin{align*}
X_{n+1}-\tilde{X}_{n+1}&=S(\Delta t)\bigl(\Phi_{\Delta t}(X_n)-\Phi_{\Delta t}(\tilde{X}_n)\bigr)\\
&+\Delta tS(\Delta t)\bigl(\Psi_{\Delta t}(\tilde{X}_n)-\Psi_{\Delta t}(X(t_n))\bigr)\\
&+\Delta tS(\Delta t)\bigl(\Psi_{\Delta t}(X(t_n))-\Psi_{0}(X(t_n))\bigr)\\
&+A^{-1}\bigl(I-S(\Delta t)\bigr)\bigl(G(X_n)-G(\tilde{X}_n)\bigr)\\
&+A^{-1}\bigl(I-S(\Delta t)\bigr)\bigl(G(\tilde{X}_n)-G(X(t_n))\bigr).
\end{align*}
Note that $\|A^{-1}\bigl(I-S(\Delta t)\bigr)\|_{\mathcal{L}(\HH)}\le \Delta t$. Owing to Lemma~\ref{lem:flow}, the mapping $\Phi_{\Delta t}$ is globally Lipschitz continuous, uniformly with respect to $\Delta t$. Recall also that $G$ is globally Lipschitz continuous. Moreover, applying the properties stated in Lemma~\ref{lem:flow} on the mapping $\Psi_{\Delta t}$, one obtains the following: there exists $C\in(0,\infty)$ such that one has
\begin{align*}
\|X_{n+1}-\tilde{X}_{n+1}\|_{\HH}&\le e^{C\Delta t}\|X_n-\tilde{X}_n\|_{\HH}\\
&+C\Delta t\bigl(1+\|\tilde{X}_n\|_{\EE}^{2q}+\|X(t_n)\|_{\EE}^{2q}\bigr)\|\tilde{X}_n-X(t_n)\|_{\HH}\\
&+C\Delta t^2\bigl(1+\|X(t_n)\|_{\EE}^{2q+1}\bigr)\\
&+C\Delta t\|X_n-\tilde{X}_n\|_{\HH}\\
&+C\Delta t\|\tilde{X}_n-X(t_n)\|_{\HH}
\end{align*}
Note that $e^{C\Delta t}+C\Delta t\le e^{2C\Delta t}$. Let $p\in\N$. Applying the Minkowski and the Cauchy--Schwarz inequalities, for all $n\in\{0,\ldots,N-1\}$ one has
\begin{align*}
\bigl(\E[\|X_{n+1}-\tilde{X}_{n+1}\|_{\HH}^{p}]\bigr)^{\frac{1}{p}}&\le e^{2C\Delta t}\bigl(\E[\|X_{n}-\tilde{X}_{n}\|_{\HH}^{p}]\bigr)^{\frac{1}{p}}\\
&+C\Delta t\Bigl(1+\bigl(\E[\|\tilde{X}_n\|_{\EE}^{4qp}]\bigr)^{\frac{1}{p}}+\bigl(\E[\|X(t_n)\|_{\EE}^{4qp}]\bigr)^{\frac{1}{2p}}\Bigr)\bigl(\E[\|\tilde{X}_n-X(t_n)\|_{\HH}^{2p}]\bigr)^{\frac{1}{2p}}\\
&+C\Delta t^2\Bigl(1+\bigl(\E[\|X(t_n)\|_{\EE}^{(2q+1)p}]\bigr) ^{\frac{1}{p}}\Bigr).
\end{align*}
Using the moment bounds~\eqref{eq:boundXtilde} for $\tilde{X}_n$ from Proposition~\ref{propo:boundXtilde} and~\eqref{eq:momentsexact} for $X(t_n)$ from Proposition~\ref{propo:exact}, and the inequality~\eqref{eq:incrementsexact} from Proposition~\ref{propo:exact}, one obtains the following upper bound: for all $n\in\{0,\ldots,N-1\}$
\[
\bigl(\E[\|X_{n+1}-\tilde{X}_{n+1}\|_{\HH}^{p}]\bigr)^{\frac{1}{p}}\le e^{2C\Delta t}\bigl(\E[\|X_{n}-\tilde{X}_{n}\|_{\HH}^{p}]\bigr)^{\frac{1}{p}}+C_{p,H,\alpha}(T,\|x_0\|_{\EE},\|x_0\|_{\HH^\alpha})\Delta t^{1+\frac{\alpha}{2}}.
\]
It is then straightforward to obtain the error estimate~\eqref{eq:errorXtildeXnum} and the proof of Proposition~\ref{propo:errorXtildeXnum} is completed.
\end{proof}

\begin{proof}[Proof of Theorem~\ref{theo:main}]
For all $n\in\{0,\ldots,N\}$, one has
\[
X(t_n)-X_n=X(t_n)-\tilde{X}_n+\tilde{X}_n-X_n.
\]
It thus suffices to combine the error estimates~\eqref{eq:errorXtildeXexact} and~\eqref{eq:errorXtildeXnum} from Propositions~\ref{propo:errorXtildeXexact} and~\ref{propo:errorXtildeXnum} respectively, in order to establish the error estimate~\eqref{eq:errorscheme}. The proof of Theorem~\ref{theo:main} is then completed.
\end{proof}

\section{Numerical experiments}\label{S.5}

In this section, we present several numerical experiments which illustrate the convergence result obtained in Theorem~\ref{theo:main}.
 
We consider the parabolic stochastic partial differential equation driven by fractional noise described as
\begin{equation*}
\left \{
\begin{aligned}
&\frac{\partial X(t,\xi)}{\partial t}=\varepsilon\frac{\partial^2X(t,\xi)}{\partial \xi^2}+f(X(t,\xi))+g(X(t,\xi))+\dot{B}^H(t,\xi), ~~t\in (0,1],~~\xi\in (0,1), ~~\varepsilon>0,\\
&X(0,\xi)=\sin(\pi \xi), ~~\xi\in (0,1),\\
&X(t,0)=X(t,1)=0, ~~t\in (0,1].
\end{aligned}\right.
\end{equation*}

In all the numerical experiments, the spectral Galerkin method is used for the spatial discretization. The spatial dimension is set equal to $N=2^k$, where $k\in\{8,9,10\}$. The expectations are estimated by the Monte Carlo by averaging over $200$ independent realizations. The final time is $T=1$.

Figures~\ref{F.1},~\ref{F.2} and~\ref{F.3} display in logarithmic scales how the error depends on the time-step size $\Delta t=T/L$, where $L=2^i$, with $i\in\{3,4,5,6,7,8,9\}$, for different values of the Hurst parameter $H\in \{0.3,0.5,0.7,0.9\}$ and of $N$. The objective is to illustrate the convergence with rate $H-1/4$ given by Theorem~\ref{theo:main}, and a second reference line with slope $H$ is also displayed. The reference solution is computed using the fully discrete scheme with time-step size $\Delta t_{\rm ref}=T/L_{\rm ref}$ with $L_{\rm ref}=2^{13}$.

First, for Figures~\ref{F.1} and~\ref{F.2}, we choose $f(X)+g(X)=-X^3+X$. One has $\varepsilon=1$ for Figure~\ref{F.1} and $\varepsilon=0.01$ for Figure~\ref{F.2}. We propose two numerical schemes: a semi-phase flow (SPF) method, with $f(X)=-X^3$ and $g(X)=X$, and a full-phase flow (FPF) method, with $f(X)=-X^3+X$ and $g(X)=0$. For all the values of the Hurst parameter $H$, we observe the convergence with order $H-1/4$, as predicted by Theorem~\ref{theo:main}, for both methods.

Second, for Figure~\ref{F.3}, we choose $f(X)=-X^3$ and $g(X)=X+\sin(X)+1$. As above, one observes the convergence with order $H-1/4$ as predicted by Theorem~\ref{theo:main}.

Note that when small values of $H$, when $\Delta t$ is small one seems to observe a higher order of convergence $H$, which is due to using a finite dimensional approximation.

%In experiments 1 and 2, we choose  $f(t,X)=-X^3+X$, the parameters $\varepsilon$ are set to $1$ and $0.01$, respectively. In experiment 3, $f(t,X)=-X^3+X+\sin(X)+1$ and $\varepsilon=1$. In all numerical experiments, the spectral Galerkin method is used for the spatial discretization with $N=2^k$, $k=8, 9, 10$, respectively. And we approximate the expectations by computing averages over 200 samples. The numerical solution is simulated by the full discretization scheme with uniform mesh size $\Delta t=T/L$, where $L=2^i$, $i=3,4,5,6,7,8,9$, respectively.  The $``$exact solution'' is referred to a numerical solution computed by the full discretization scheme with much smaller uniform mesh size $\Delta t=T/M$, where $M=2^{13}$. Figures \ref{F.1}, \ref{F.2}, and \ref{F.3} demonstrate the temporal approximation errors ($\log$-$\log$ scale) in $L^2(\Omega;{\mathbb H})$ sense at the endpoint $T=1$ with $L=2^i$, $i=3,4,5,6,7,8,9$ for different Hurst parameters $H$ and $\varepsilon$. 

\begin{figure}[!thb]
	\centering
	{\includegraphics[height=2.5in, width=2.5in]{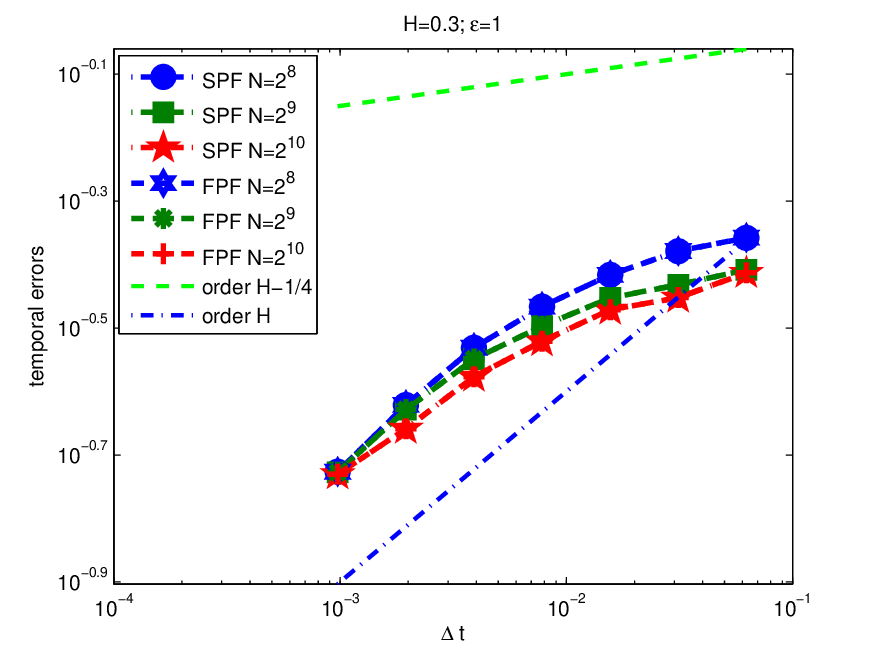}}
	{\includegraphics[height=2.5in, width=2.5in]{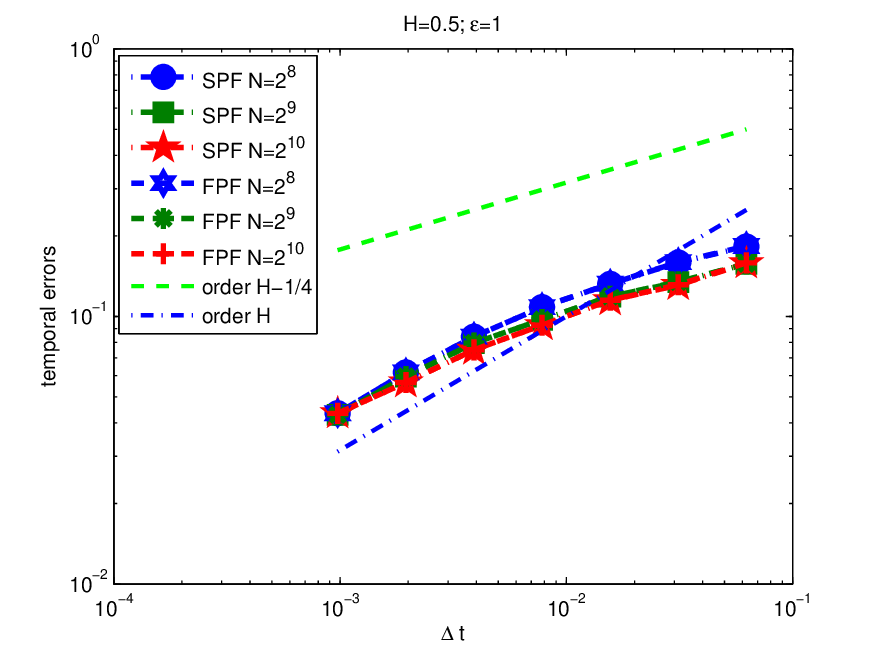}}\\
	{\includegraphics[height=2.5in, width=2.5in]{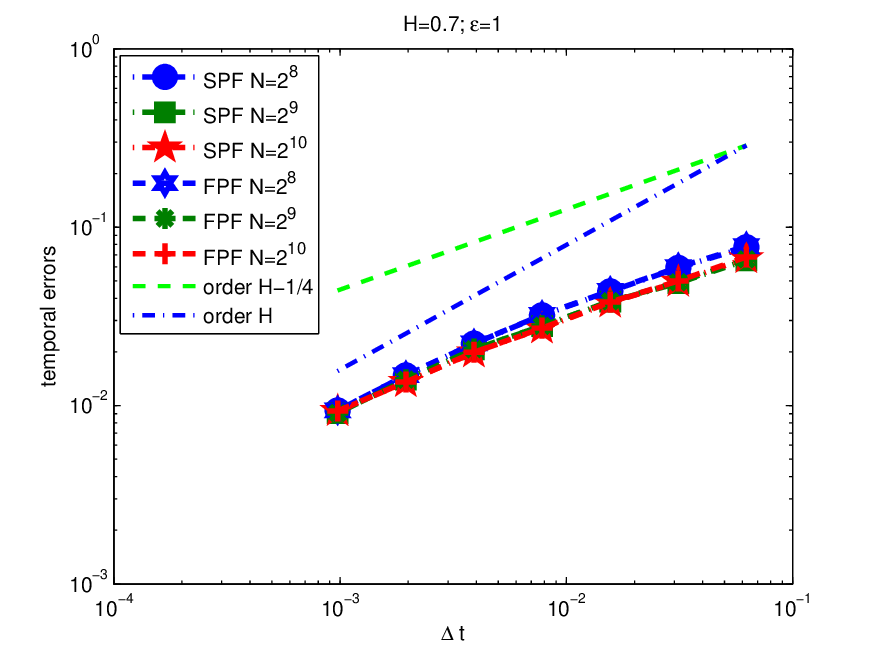}}
	{\includegraphics[height=2.5in, width=2.5in]{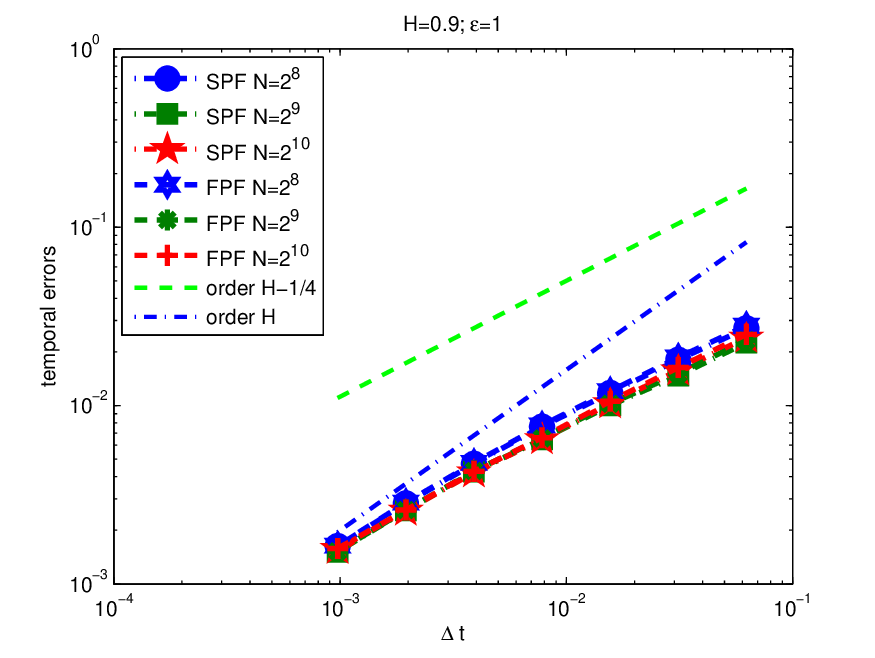}}
	\caption{Temporal convergence rate of numerical approximations with $f(t,X)=-X^3+X$ and $\varepsilon=1$ for different Hurst parameters $H\in \{0.3,0.5,0.7,0.9\}$.}
	\label{F.1}
\end{figure}

\begin{figure}[!thb]
	\centering
	{\includegraphics[height=2.5in, width=2.5in]{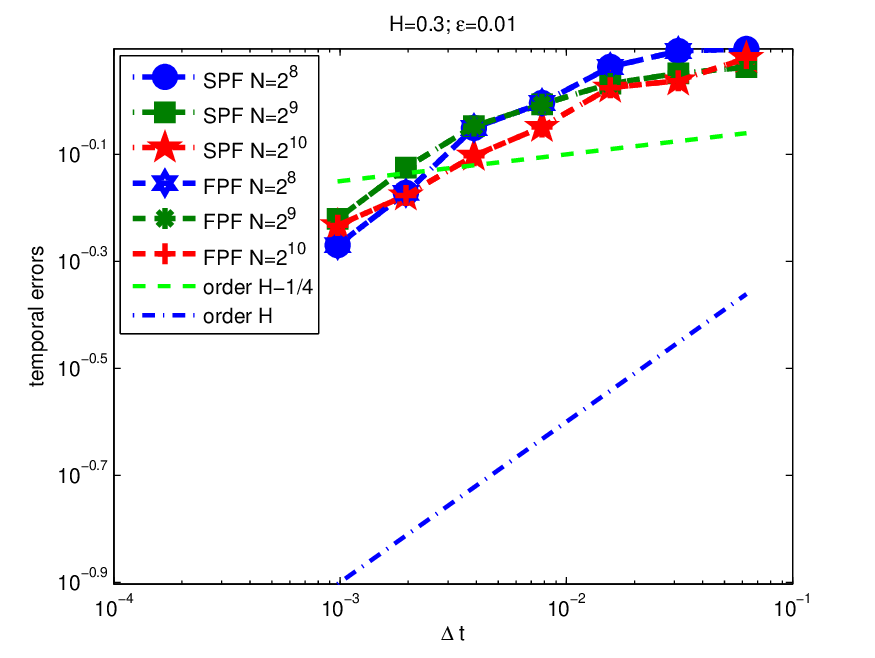}}
	{\includegraphics[height=2.5in, width=2.5in]{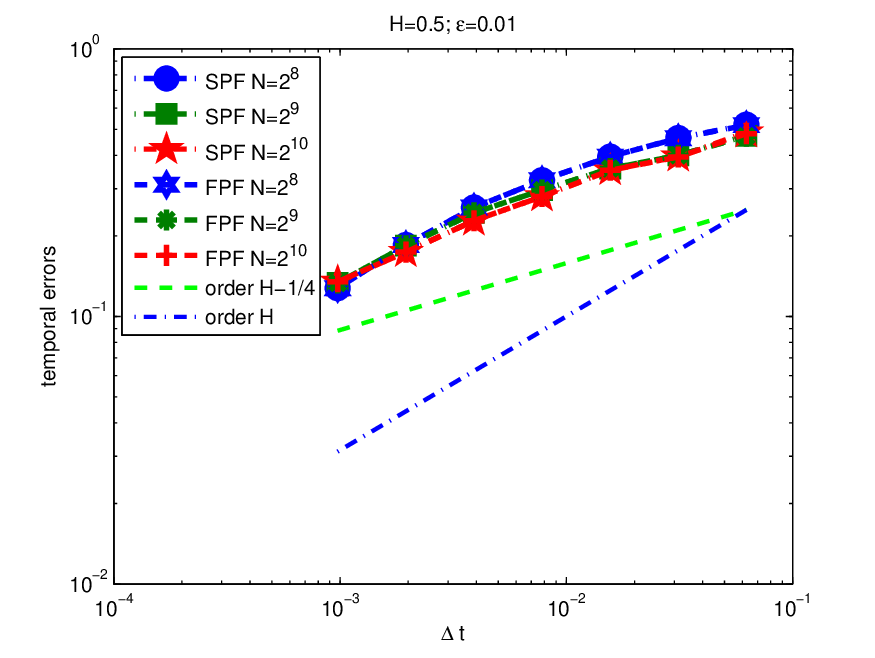}}\\
	{\includegraphics[height=2.5in, width=2.5in]{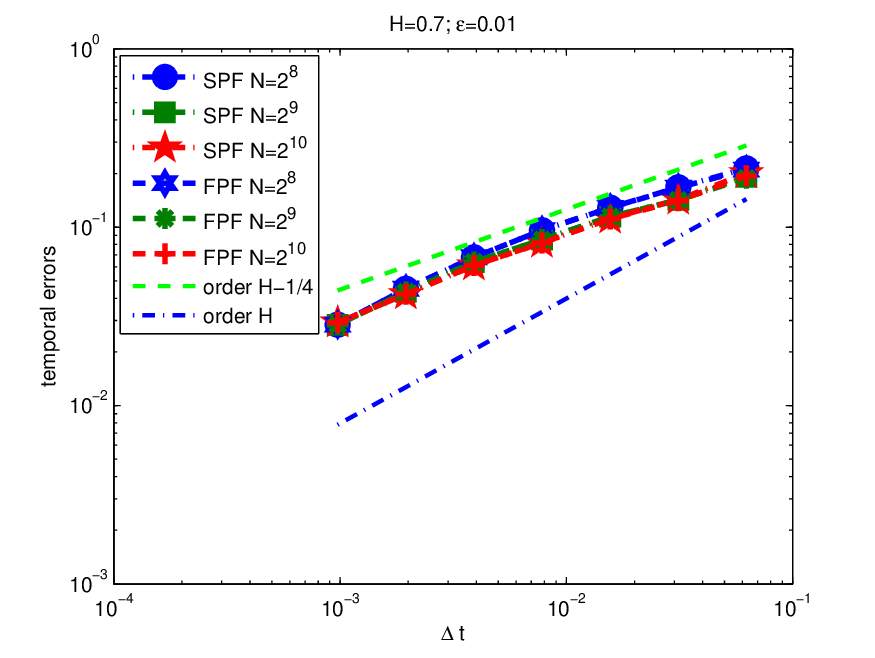}}
	{\includegraphics[height=2.5in, width=2.5in]{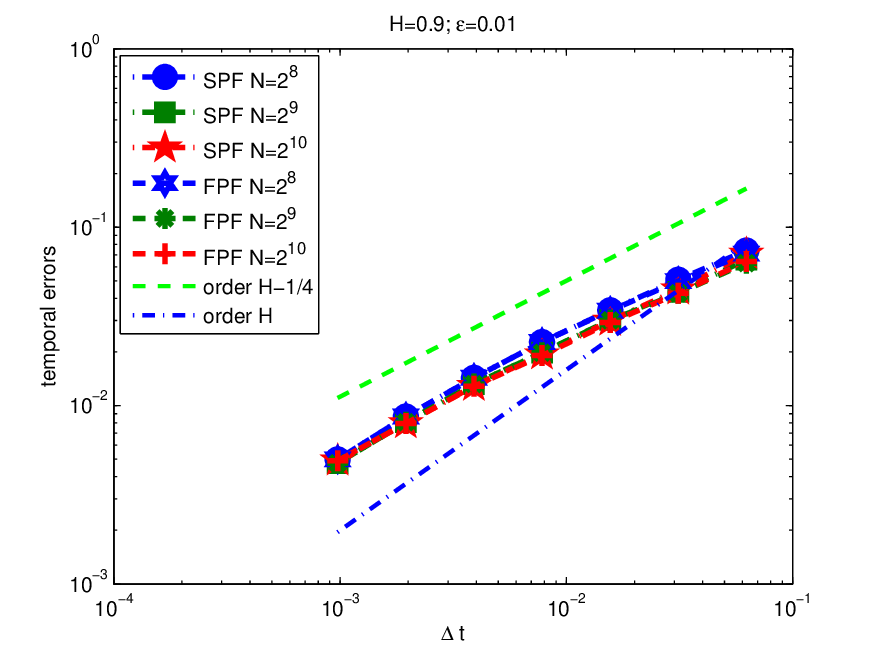}}
	\caption{Temporal convergence rate of numerical approximations with $f(t,X)=-X^3+X$ and $\varepsilon=0.01$ for different Hurst parameters $H\in \{0.3,0.5,0.7,0.9\}$.}
	\label{F.2}
\end{figure}

\begin{figure}[!thb]
	\centering
	{\includegraphics[height=2.5in, width=2.5in]{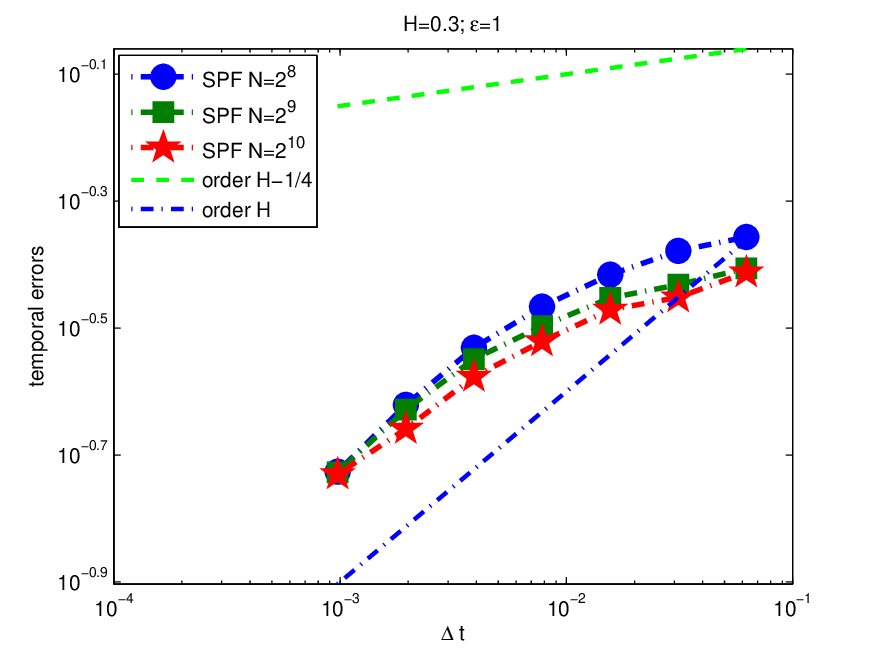}}
	{\includegraphics[height=2.5in, width=2.5in]{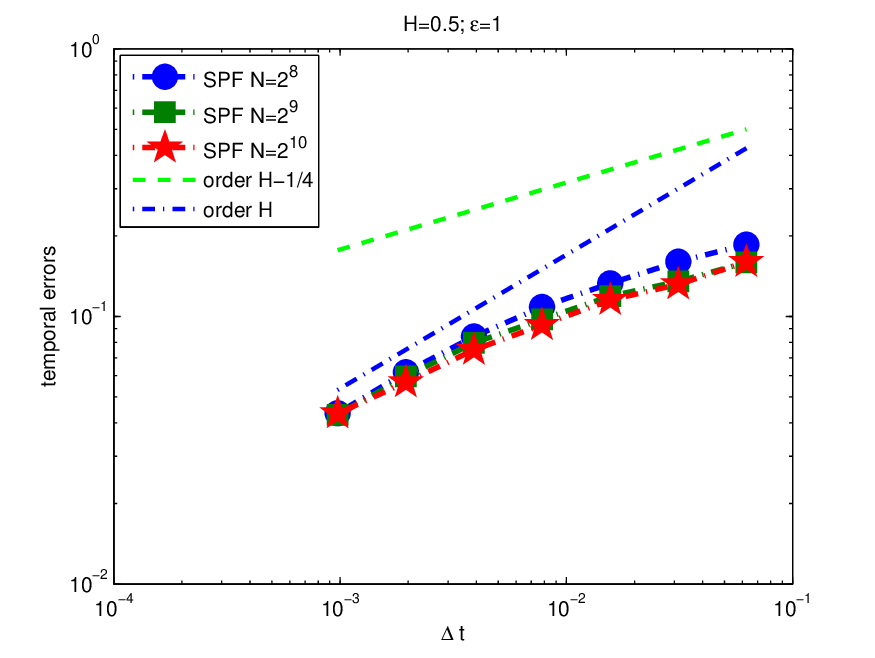}}\\
	{\includegraphics[height=2.5in, width=2.5in]{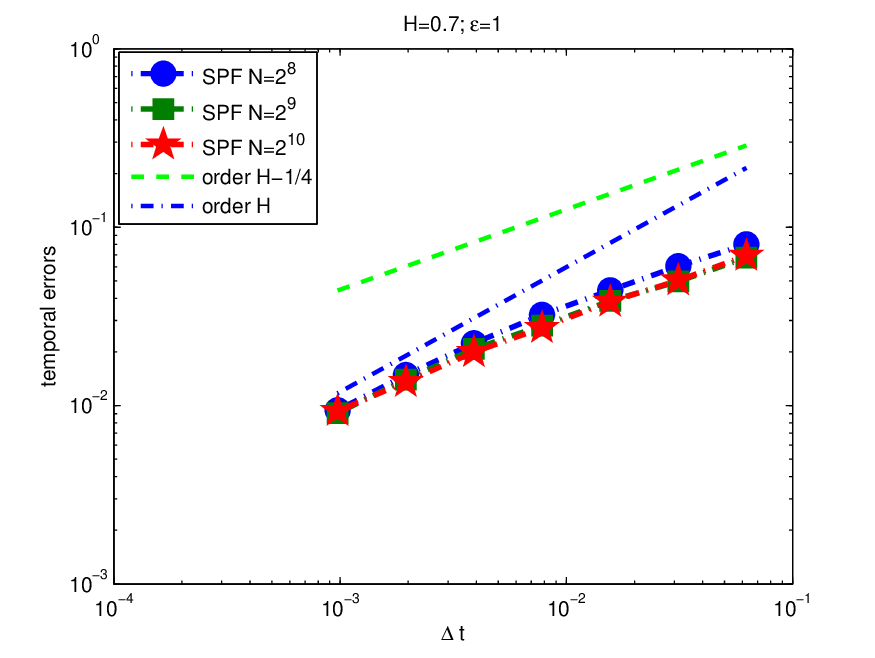}}
	{\includegraphics[height=2.5in, width=2.5in]{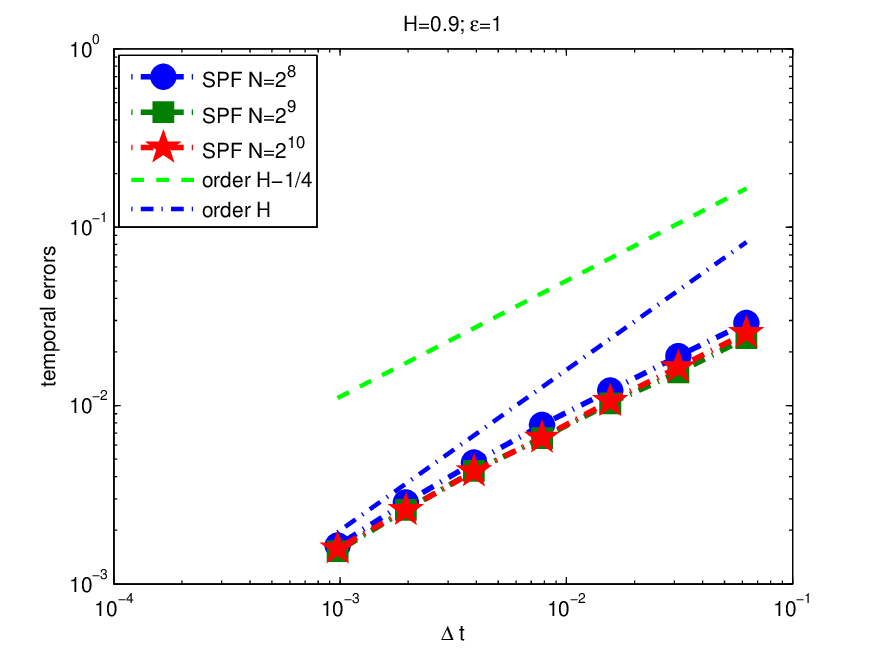}}
	\caption{Temporal convergence rate of numerical approximations with $f(t,X)=-X^3+X+\sin(X)+1$ and $\varepsilon=1$ for different Hurst parameters $H\in \{0.3,0.5,0.7,0.9\}$.}
	\label{F.3}
\end{figure}

%Based on the numerical simulation results, we observe the following findings:
%
%(i) In experiments 1 and 2, we compare our proposed semi-phase flow (SPF) method with the splitting method based on the full-phase flow (FPF). In experiments 1 and 2, we can notice that for all $\varepsilon>0$ the two convergence order curves have nearly overlapped. In other words, for $\varepsilon>0$, the proposed method demonstrates nearly identical convergence order to the splitting method based on the full-phase flow. 
%
%(ii) From experiments 1, 2, and 3, we note that for all the values of $H\in(1/4, 1)$ and $\varepsilon>0$, the convergence order of the proposed method basiclly agrees with the theoretical onvergence order $H-1/4$ that is proved in Theorem  \ref{theo:main}(uniform in $N$). In particular, when $H$ is close to 1, its strong convergence order tends to $0.75$, while when $H$ is close to 0.25, its strong convergence rate is relatively slow. This phenomenon indicates that 
%the smoother the path of the fractional noise (the higher the value of $H$), the faster the convergence speed. Conversely, the rougher the path of the fractional noise (the lower the value of $H$),  the slower the convergence speed. 

\section{Conclusion}\label{S.6}

In this work, we have developed and analyzed a new class of splitting schemes based on a partial exact solution of nonlinear term to solve stochastic partial differential equations with non-Lipschitz coefficients driven by fractional noise when $1/4<H<1$. The proposed method is computationally tractable and effective. We have proved that the scheme converges with order $H-1/4$. This is illustrated by numerical experiments.

However, there are remaining open questions for future work. For example, one may study the weak order of convergence of the proposed numerical approximation. One may also investigate the construction of higher order methods, especially when $1/4<H<1/2$. Finally, in this manuscript we have only studied parabolic semilinear stochastic partial differential equations driven by a cylindrical fractional Brownian motion on the one-dimensional domain $(0,1)$. In future works, one may investigate the case of equations on multidimensional domains $(0,1)^d$, driven by other classes of noises.

\begin{appendix}

\section{Appendix: notation}\label{App.A}
Let $\bigl(\beta^H(t)\bigr)_{t\ge 0}$ be a real-valued fractional Brownian motion with Hurst index $H\in(0,1)$, given by
\[
\beta^H(t)=\int_{0}^{t}\mathcal{K}_H(t,s)\dd \beta(s),\quad \forall~t\ge 0,
\]
where $\bigl(\beta(t)\bigr)_{t\ge 0}$ is a standard real-valued Brownian motion.

For all $\lambda\in(0,\infty)$, let us consider the real-valued fractional Ornstein--Uhlenbeck process $\bigl(\Z^{H,\lambda}(t)\bigr)_{t\ge 0}$ defined by
\begin{equation}\label{eq:ZHlambda_def}
\Z^{H,\lambda}(t)=\int_{0}^{t}e^{-\lambda(t-s)}\dd \beta^H(s),\quad \forall~t\ge 0.
\end{equation}

For all $\lambda\in(0,\infty)$ and all $t\ge 0$, introduce the auxiliary function $\phi_{\lambda,t}$ defined by
\[
\phi_{\lambda,t}(s)=e^{-\lambda(t-s)}\mathds{1}_{s\in[0,t]}.
\]
Then for all $t\ge 0$ one has
\begin{equation}\label{eq:ZHlambda}
\Z^{H,\lambda}(t)=\int_{0}^{t}\phi_{\lambda,t}(s)\dd \beta^H(s).
\end{equation}
Moreover, for all $\lambda\in(0,\infty)$ and all $t_2\ge t_1\ge 0$, introduce the auxiliary function $\psi_{\lambda,t_1,t_2}$ defined by 
\[
\psi_{\lambda,t_1,t_2}(s)=e^{-\lambda(t_2-s)}\mathds{1}_{s\in[t_1,t_2]}.
\]
Then for all $t\ge 0$, one has
\begin{align}
\Z^{H,\lambda}(t_2)-\Z^{H,\lambda}(t_1)&=\bigl(e^{-\lambda(t_2-t_1)}-1\bigr)\Z^{H,\lambda}(t_1)+\int_{t_1}^{t_2}e^{-\lambda(t_2-s)}\dd\beta^H(s)\\
&=\bigl(e^{-\lambda(t_2-t_1)}-1\bigr)\Z^{H,\lambda}(t_1)+\int_{0}^{t_2}\psi_{\lambda,t_1,t_2}(s)\dd \beta^H(s).\label{eq:incrementsZHlambda}
\end{align}

Since the case $H=1/2$ is well-known, we assume that $H\neq 1/2$. The cases $H\in(1/2,1)$ and $H\in(0,1/2)$ are treated separately, since different techniques are employed in the proofs.

Given the time-step size $\Delta t\in(0,1)$, recall that $t_n=n\Delta t$ for all $n\in\N_0$. Let the increments of the real-valued fractional Brownian motion $\bigl(\beta^H(t)\bigr)_{t\ge 0}$ on the interval $[t_n,t_{n+1}]$  be denoted by $\Delta \beta_n^H=\beta^{H}(t_{n+1})-\beta^{H}(t_n)$, for all $n\ge 0$. Moreover, introduce the auxiliary mapping $\ell:[0,\infty)\to \N_0$, defined by $\ell(t)=\lfloor t/\Delta t\rfloor$, where $\lfloor\cdot\rfloor$ denotes the integer part, i.e. one has $t_{\ell(s)}=t_n$ for all $s\in[t_n,t_{n+1})$.

For all $\lambda\in(0,\infty)$, let us consider the discrete-time approximation $\bigl(\Z_{n}^{H,\lambda}\bigr)_{n\ge 0}$ of the real-valued fractional Ornstein--Uhlenbeck process defined by~\eqref{eq:ZHlambda_def}: set $\Z_0^{H,\lambda}=0$, and for all $n\ge 0$
\begin{equation}\label{eq:ZnHlambda_def}
\Z_{n+1}^{H,\lambda}=e^{-\lambda\Delta t}\bigl(\Z_n^{H,\lambda}+\Delta \beta_n^H\bigr).
\end{equation}
One obtains the following expression: for all $n\ge 0$ one has
\begin{equation}\label{eq:ZnHlambda}
\Z_n^{H,\lambda}=\sum_{j=0}^{n-1}\int_{t_j}^{t_{j+1}}e^{-\lambda(t_n-t_j)}\dd \beta^H(s)=\int_{0}^{t_n}e^{-\lambda(t_n-t_{\ell(s)})}\dd \beta^H(s).
\end{equation}
For all $\lambda\in(0,\infty)$, $\Delta t>0$ and $n\ge 0$, introduce the auxiliary function $\varepsilon_{\lambda,t_n}^{\Delta t}$ defined by
\[
\varepsilon_{\lambda,t_n}^{\Delta t}(s)=\Bigl(e^{-\lambda(t_n-s)}-e^{-\lambda(t_n-t_{\ell(s)})}\Bigr)\mathds{1}_{[0,t_n]}(s).
\]
Combining~\eqref{eq:ZHlambda_def} and~\eqref{eq:ZnHlambda}, the error $\Z^{H,\lambda}(t_n)-\Z_n^{H,\lambda}$ is expressed as
\begin{equation}\label{eq:errorZnHlambda}
\Z^{H,\lambda}(t_n)-\Z_n^{H,\lambda}=\int_{0}^{t_n}\varepsilon_{\lambda,t_n}^{\Delta t}(s)\dd \beta^H(s).
\end{equation}

\section{Appendix: properties of the one-dimensional fractional Ornstein--Uhlenbeck process when $H\in(1/2,1)$}\label{app:+}

\subsection{Moment bounds}

\begin{lemma}\label{lem:moment1DH+}%[Moment bounds for $H\in(1/2,1)$]
For all $H\in(1/2,1)$, there exists $C_H\in(0,\infty)$ such that for all $\lambda\in(0,\infty)$
\begin{equation}\label{eq:moment1DH+}
\underset{t\ge 0}\sup~\E[|\Z^{H,\lambda}(t)|^2]\le C_H\lambda^{-2H}.
\end{equation}
\end{lemma}

\begin{proof}[Proof of Lemma~\ref{lem:moment1DH+}]
Using the identity~\eqref{eq:ZHlambda} and applying the It\^o isometry formula~\eqref{eq:Ito} and its formulation~\eqref{eq:Ito+} when $H\in(1/2,1)$, one has
\begin{align*}
\E[|\Z^{H,\lambda}(t)|^2]&=\int_{0}^{t}|\mathcal{K}_{H,t}^\star\phi_{\lambda,t}(s)|^2 \dd s\\
&=H(2H-1)\int_{0}^{t}\int_{0}^{t}\phi_{\lambda,t}(s_1)\phi_{\lambda,t}(s_2)|s_2-s_1|^{2H-2}\dd s_1 \dd s_2\\
&=H(2H-1)\int_{0}^{t}\int_{0}^{t}e^{-\lambda(t-s_1)}e^{-\lambda(t-s_2)}|s_2-s_1|^{2H-2}\dd s_1 \dd s_2.
\end{align*}
The integral appearing in the right-hand side above can be written as
\begin{align*}
\int_{0}^{t}\int_{0}^{t}e^{-\lambda(t-s_1)}e^{-\lambda(t-s_2)}|s_2-s_1|^{2H-2}\dd s_1 \dd s_2&=\int_{0}^{t}\int_{0}^{t}e^{-\lambda s_1}e^{-\lambda s_2}|s_2-s_1|^{2H-2}\dd s_1 \dd s_2\\
&=\lambda^{-2H}\int_{0}^{\lambda t}\int_{0}^{\lambda t}e^{-r_1}e^{-r_2}|r_2-r_1|^{2H-2}\dd r_1 \dd r_2,
\end{align*}
using changes of variables $s_1=\lambda r_1$ and $s_2=\lambda r_2$.

Applying the Fubini theorem, one obtains
\begin{align*}
\int_{0}^{\lambda t}\int_{0}^{\lambda t}e^{-r_1}e^{-r_2}|r_2-r_1|^{2H-2}\dd r_1 \dd r_2&=2\int_{0}^{\lambda t}\int_{r_1}^{\lambda t}e^{-2r_1}e^{-(r_2-r_1)}(r_2-r_1)^{2H-2}\dd r_2 \dd r_1\\
&=2 \int_{0}^{\lambda t}e^{-2r_1} \int_{0}^{\lambda (t-r_1)} e^{-r}r^{2H-2} \dd r \dd r_1 \\
&\le \int_{0}^{+\infty}e^{-r}r^{2H-2}\dd r<\infty.
\end{align*}
Due to the condition $H>1/2$ the integral in the right-hand side above is finite. Therefore there exists $C_H\in(0,\infty)$ such that the inequality~\eqref{eq:moment1DH+} holds, and the proof of Lemma~\ref{lem:moment1DH+} is completed.
\end{proof}

\subsection{Temporal regularity}

\begin{lemma}\label{lem:temp1DH+}%[Temporal regularity for $H\in(1/2,1)$]
For all $H\in(1/2,1)$ and all $\alpha\in[0,2H]$, there exists $C_{\alpha,H}\in(0,\infty)$ such that for all $\lambda\in(0,\infty)$, and all $t_1,t_2\ge 0$, one has
\begin{equation}\label{eq:temp1DH+}
\E[|\Z^{H,\lambda}(t_2)-\Z^{H,\lambda}(t_1)|^2]\le C_H\lambda^{\alpha-2H}|t_2-t_1|^\alpha.
\end{equation}
\end{lemma}

\begin{proof}[Proof of Lemma~\ref{lem:temp1DH+}]
Recall the decomposition~\eqref{eq:incrementsZHlambda} of $\Z^{H,\lambda}(t_2)-\Z^{H,\lambda}(t_1)$.

For the first term in~\eqref{eq:incrementsZHlambda}, using the inequality
\[
|1-e^{-z}|\le z^{\gamma},\qquad \forall~z\ge 0,\quad \forall~\gamma\in[0,1],
\]
and applying the moment bounds~\eqref{eq:moment1DH+} from Lemma~\ref{lem:moment1DH+}, one has
\begin{equation}\label{eq:aux1tempH+}
\E[\big|\bigl(e^{-\lambda(t_2-t_1)}-1\bigr)\Z^{H,\lambda}(t_1)\big|^2]\le C_H\lambda^{\alpha-2H}(t_2-t_1)^{\alpha}.
\end{equation}
For the second term in~\eqref{eq:incrementsZHlambda}, applying the It\^o isometry formula~\eqref{eq:Ito} and its formulation~\eqref{eq:Ito+} when $H\in(1/2,1)$, one has
\begin{align*}
\E[\big|\int_{t_1}^{t_2}e^{-\lambda(t_2-s)}\dd\beta^H(s)\big|^2]&=\E[\big|\int_{0}^{t_2}\psi_{\lambda,t_1,t_2}(s)\dd\beta^H(s)\big|^2]\\
&=\int_{0}^{t_2}|\mathcal{K}_{H,t_2}^\star\psi_{\lambda,t_1,t_2}(s)|^2\dd s\\
&=H(2H-1)\int_{0}^{t_2}\int_{0}^{t_2}\psi_{\lambda,t_1,t_2}(s_1)\psi_{\lambda,t_1,t_2}(s_2)|s_2-s_1|^{2H-2}\dd s_1 \dd s_2\\
&=H(2H-1)\int_{t_1}^{t_2}\int_{t_1}^{t_2}e^{-\lambda(t_2-s_1)}e^{-\lambda(t_2-s_2)}|s_2-s_1|^{2H-2}\dd s_1 \dd s_2\\
&=H(2H-1)\int_{0}^{t_2-t_1}\int_{0}^{t_2-t_1}e^{-\lambda s_1}e^{-\lambda s_2}|s_2-s_1|^{2H-2}\dd s_1 \dd s_2.
\end{align*}
Using changes of variables $r_1=\lambda s_1$ and $r_2=\lambda s_2$, one obtains
\begin{align*}
\int_{0}^{t_2-t_1}\int_{0}^{t_2-t_1}&e^{-\lambda s_1}e^{-\lambda s_2}|s_2-s_1|^{2H-2}\dd s_1 \dd s_2\\
&=\lambda^{-2H}\int_{0}^{\lambda(t_2-t_1)}\int_{0}^{\lambda(t_2-t_1)}e^{-r_1}e^{- r_2}|r_2-r_1|^{2H-2}\dd r_1 \dd r_2\\
&=2\lambda^{-2H}\int_{0}^{\lambda(t_2-t_1)}\int_{r_1}^{\lambda(t_2-t_1)}e^{-2r_1}e^{-(r_2-r_1)}|r_2-r_1|^{2H-2}\dd r_2 \dd r_1\\
&=2\lambda^{-2H}\int_{0}^{\lambda(t_2-t_1)}e^{-2r_1}\dd r_1\int_{0}^{\lambda(t_2-t_1)}e^{-r}r^{2H-2}\dd r.
\end{align*}
To proceed, one needs to treat separately the cases $\lambda(t_2-t_1)\le 1$ and $\lambda(t_2-t_1)>1$.

On the one hand, when $\lambda(t_2-t_1)\le 1$, one has
\begin{align*}
\int_{0}^{\lambda(t_2-t_1)}e^{-2r_1}\dd r_1\int_{0}^{\lambda(t_2-t_1)}e^{-r}r^{2H-2}\dd r&\le \int_{0}^{\lambda(t_2-t_1)}\dd r_1\int_{0}^{\lambda(t_2-t_1)}r^{2H-2}\dd r\\
&\le(\lambda(t_2-t_1)\bigr) \frac{\bigl(\lambda(t_2-t_1)\bigr)^{2H-1}}{2H-1}=\frac{\bigl(\lambda(t_2-t_1)\bigr)^{2H}}{2H-1}.
\end{align*}
On the other hand, when $\lambda(t_2-t_1)>1$, one has
\[
\int_{0}^{\lambda(t_2-t_1)}e^{-2r_1}\dd r_1\int_{0}^{\lambda(t_2-t_1)}e^{-r}r^{2H-2}\dd r\le \int_{0}^{\infty}e^{-2r_1}\dd r_1\int_{0}^{\infty}e^{-r}r^{2H-2}\dd r<\infty.
\]
Combining the two cases, there exist $C_{0,H},C_{1,H}\in(0,\infty)$ such that one has the two inequalities
\begin{align*}
\int_{0}^{\lambda(t_2-t_1)}e^{-2r_1}\dd r_1\int_{0}^{\lambda(t_2-t_1)}e^{-r}r^{2H-2}\dd r&\le C_{0,H}\\
&\le C_{1,H}\bigl(\lambda(t_2-t_1)\bigr)^{2H}.
\end{align*}
By an interpolation interpolation, for all $\alpha\in[0,2H]$, there exists $C_{\alpha,H}\in(0,\infty)$ such that one has
\begin{equation}\label{eq:aux2tempH+}
\int_{0}^{\lambda(t_2-t_1)}e^{-2r_1}\dd r_1\int_{0}^{\lambda(t_2-t_1)}e^{-r}r^{2H-2}\dd r\le C_{\alpha,H}\bigl(\lambda(t_2-t_1)\bigr)^{\alpha}.
\end{equation}
Gathering the estimates~\eqref{eq:aux1tempH+} and~\eqref{eq:aux2tempH+}, and recalling the decomposition~\eqref{eq:incrementsZHlambda}, one obtains~\eqref{eq:temp1DH+} and the proof of Lemma~\ref{lem:temp1DH+} is completed.
\end{proof}

\section{Appendix: properties of the one-dimensional fractional Ornstein--Uhlenbeck process when $H\in(0,1/2)$}\label{app:-}

\subsection{Moment bounds}

\begin{lemma}\label{lem:moment1DH-}%[Moment bounds for $H\in(0,\frac12)$]
For all $H\in(0,\frac12)$, there exists $C_H\in(0,\infty)$ such that for all $\lambda\in(0,\infty)$
\begin{equation}\label{eq:moment1DH-}
\underset{t\ge 0}\sup~\E[|\Z^{H,\lambda}(t)|^2]\le C_H\lambda^{-2H}.
\end{equation}
\end{lemma}

\begin{proof}[Proof of Lemma~\ref{lem:moment1DH-}]
Using the identity~\eqref{eq:ZHlambda}, applying the It\^o isometry formula~\eqref{eq:Ito}, and using the expression~\eqref{eq:KH-} for $\mathcal{K}_{H,T}^\star\phi (t)$ when $H\in(0,\frac12)$, one has
\begin{align*}
\E[|\Z^{H,\lambda}(t)|^2]&=\int_{0}^{t}|\mathcal{K}_{H,t}^\star\phi_{\lambda,t}(s)|^2 \dd s\\
&\le C\int_{0}^{t}\frac{|\phi_{\lambda,t}(s)|^2}{(t-s)^{1-2H}}\dd s+C\int_{0}^{t}\Big|s^{\frac12-H}\int_{s}^{t}\frac{s^{H-\frac12}\phi_{\lambda,t}(s)-\tau^{H-\frac12}\phi_{\lambda,t}(\tau)}{(\tau-s)^{\frac32-H}}\dd \tau\Big|^2 \dd s.
\end{align*}
For the first term, for all $t\ge 0$ one has
\begin{align*}
\int_{0}^{t}\frac{|\phi_{\lambda,t}(s)|^2}{(t-s)^{1-2H}}\dd s&=\int_{0}^{t}\frac{e^{-2\lambda(t-s)}}{(t-s)^{1-2H}}\dd s=\int_{0}^{t}\frac{e^{-2\lambda s}}{s^{1-2H}}\dd s\\
&=\lambda^{-2H}\int_{0}^{\lambda t}\frac{e^{-2r}}{r^{1-2H}}\dd r\le \lambda^{-2H}\int_{0}^{\infty}\frac{e^{-2r}}{r^{1-2H}}\dd r,
\end{align*}
with a change of variable $r=\lambda s$. Owing to the condition $H>0$ one obtains the following result: there exists $C_H\in(0,\infty)$ such that for all $\lambda\in(0,\infty)$ and $t\ge 0$ one has
\begin{equation}\label{eq:auxmoments0}
\int_{0}^{t}\frac{|\phi_{\lambda,t}(s)|^2}{(t-s)^{1-2H}}\dd s\le C_H\lambda^{-2H}.
\end{equation}

For the second term, note that one has the decomposition
\[
s^{H-\frac12}\phi_{\lambda,t}(s)-\tau^{H-\frac12}\phi_{\lambda,t}(\tau)=\bigl(s^{H-\frac12}-\tau^{H-\frac12}\bigr)\phi_{\lambda,t}(s)+\tau^{H-\frac12}\bigl(\phi_{\lambda,t}(s)-\phi_{\lambda,t}(\tau)\bigr),
\]
therefore one obtains the upper bound
\[
\int_{0}^{t}\Big|s^{\frac12-H}\int_{s}^{t}\frac{s^{H-\frac12}\phi_{\lambda,t}(s)-\tau^{H-\frac12}\phi_{\lambda,t}(\tau)}{(\tau-s)^{\frac32-H}}\dd \tau\Big|^2 \dd s\le 2\bigl(\mathcal{I}_{H,\lambda,t}^1+\mathcal{I}_{H,\lambda,t}^2\bigr),
\]
where $\mathcal{I}_{H,\lambda,t}^1$ and $\mathcal{I}_{H,\lambda,t}^2$ are defined as
\begin{align*}
\mathcal{I}_{H,\lambda,t}^1&=\int_{0}^{t}\Big|s^{\frac12-H}\phi_{\lambda,t}(s)\int_{s}^{t}\frac{s^{H-\frac12}-\tau^{H-\frac12}}{(\tau-s)^{\frac32-H}}\dd \tau\Big|^2 \dd s\\
&=\int_{0}^{t}\Big|s^{\frac12-H}e^{-\lambda(t-s)}\int_{s}^{t}\frac{s^{H-\frac12}-\tau^{H-\frac12}}{(\tau-s)^{\frac32-H}}\dd \tau\Big|^2 \dd s\\
\mathcal{I}_{H,\lambda,t}^2&=\int_{0}^{t}\Big|s^{\frac12-H}\int_{s}^{t}\frac{\tau^{H-\frac12}\bigl(\phi_{\lambda,t}(s)-\phi_{\lambda,t}(\tau)\bigr)}{(\tau-s)^{\frac32-H}}\dd \tau\Big|^2 \dd s\\
&=\int_{0}^{t}\Big|s^{\frac12-H}\int_{s}^{t}\frac{\tau^{H-\frac12}\bigl(e^{-\lambda(t-s)}-e^{-\lambda(t-\tau)}\bigr)}{(\tau-s)^{\frac32-H}}\dd \tau\Big|^2 \dd s.
\end{align*}

To deal with the term $\mathcal{I}_{H,\lambda,t}^1$, we prove the following auxiliary inequality: there exists $C_H\in(0,\infty)$ such that one has
\begin{equation}\label{eq:claim}
\int_{s}^{t}\frac{s^{H-\frac12}-\tau^{H-\frac12}}{(\tau-s)^{\frac32-H}}\dd \tau\le C_Hs^{2H-1},\quad \forall~t\ge s\ge 0.
\end{equation}
The inequality~\eqref{eq:claim} is proved as follows: using a change of variable $\tau=s(1+r)$, one has
\[
\int_{s}^{t}\frac{s^{H-\frac12}-\tau^{H-\frac12}}{(\tau-s)^{\frac32-H}}\dd \tau=\int_{0}^{\frac{t-s}{s}}s^{H-\frac12}\frac{1-(1+r)^{H-\frac12}}{s^{\frac32-H}r^{\frac32-H}} s\dd r\le s^{2H-1}\int_{0}^{\infty}\frac{1-(1+r)^{H-\frac12}}{r^{\frac32-H}} \dd r.
\]
Owing to the condition $H\in(0,\frac12)$ one has
\begin{align*}
\int_{0}^{\infty}\frac{1-(1+r)^{H-\frac12}}{r^{\frac32-H}} \dd r&\le \int_{0}^{1}\frac{1-(1+r)^{H-\frac12}}{r^{\frac32-H}} \dd r+\int_{1}^{\infty}\frac{1}{r^{\frac32-H}} \dd r\\
&\le(\frac12-H)\int_{0}^{1}\frac{1}{r^{\frac12-H}} \dd r+\int_{1}^{\infty}\frac{1}{r^{\frac32-H}} \dd r<\infty,
\end{align*}
and the proof of the inequality~\eqref{eq:claim} is completed.

Applying the inequality~\eqref{eq:claim}, for the term $\mathcal{I}_{H,\lambda,t}^1$ one then obtains the upper bound
\begin{align*}
\int_{0}^{t}\Big|s^{\frac12-H}e^{-\lambda(t-s)}\int_{s}^{t}\frac{s^{H-\frac12}-\tau^{H-\frac12}}{(\tau-s)^{\frac32-H}}\dd \tau\Big|^2 \dd s&\le C_H\int_{0}^{t}s^{2H-1}e^{-2\lambda(t-s)}\dd s.
\end{align*}
Moreover, by a change of variable $\lambda s=r$, one has
\[
\int_{0}^{t}s^{2H-1}e^{-2\lambda(t-s)}\dd s=\lambda^{-2H}\int_{0}^{\lambda t}r^{2H-1}e^{-2(\lambda t-r)}\dd r.
\]
Observe that the integral in the right-hand side above satisfies for all $t\ge 0$
\begin{align*}
\int_{0}^{\lambda t}r^{2H-1}e^{-2(\lambda t-r)}\dd r&\le \int_{0}^{1}r^{2H-1}\dd r+\mathds{1}_{\lambda t\ge 1}\int_{1}^{\lambda t}e^{-2(\lambda t-r)}\dd r\\
&\le \int_{0}^{1}r^{2H-1}\dd r+\int_{0}^{\infty}e^{-2r'}\dd r'<\infty.
\end{align*}
As a result, there exists $C_H\in(0,\infty)$ such that for all $\lambda\in(0,\infty)$ and $t\ge 0$ one has
\begin{equation}\label{eq:auxmoments1}
\mathcal{I}_{H,\lambda,t}^{1}=\int_{0}^{t}\Big|s^{\frac12-H}e^{-\lambda(t-s)}\int_{s}^{t}\frac{s^{H-\frac12}-\tau^{H-\frac12}}{(\tau-s)^{\frac32-H}}\dd \tau\Big|^2 \dd s\le C_H\lambda^{-2H}.
\end{equation}
It remains to deal with the term $\mathcal{I}_{H,\lambda,t}^2$. First, one has the upper bound
\[
\mathcal{I}_{H,\lambda,t}^2=\int_{0}^{t}\Big|s^{\frac12-H}\int_{s}^{t}\frac{\tau^{H-\frac12}\bigl(e^{-\lambda(t-s)}-e^{-\lambda(t-\tau)}\bigr)}{(\tau-s)^{\frac32-H}}\dd \tau\Big|^2 \dd s\le \int_{0}^{t}\Big|\int_{s}^{t}\frac{\bigl(e^{-\lambda(t-\tau)}-e^{-\lambda(t-s)}\bigr)}{(\tau-s)^{\frac32-H}}\dd \tau\Big|^2 \dd s.
\]
By changes of variables $u=t-s$ and $v=t-\tau$, one has
\begin{align*}
\int_{0}^{t}\Big|\int_{s}^{t}\frac{\bigl(e^{-\lambda(t-\tau)}-e^{-\lambda(t-s)}\bigr)}{(\tau-s)^{\frac32-H}}\dd \tau\Big|^2 \dd s&=\int_{0}^{t}\Big|\int_{0}^{u}\frac{e^{-\lambda v}-e^{-\lambda u}}{(u-v)^{\frac32-H}}\dd v\Big|^2\dd u\\
&=\int_{0}^{t}\Big|\int_{0}^{u}e^{-\lambda (u-w)}\frac{1-e^{-\lambda w}}{w^{\frac32-H}}\dd w\Big|^2\dd u,
\end{align*}
where the last identity is obtained by the change of variable $w=v-u$ and writing
\[
e^{-\lambda v}-e^{-\lambda u}=e^{-\lambda (u-w)}-e^{-\lambda u}=e^{-\lambda(u-w)}(1-e^{-\lambda w}).
\]

Applying the Cauchy--Schwarz inequality and the Fubini theorem, one has
\begin{align*}
\int_{0}^{t}\Big|\int_{0}^{u}e^{-\lambda (u-w)}\frac{1-e^{-\lambda w}}{w^{\frac32-H}}\dd w\Big|^2\dd u&\le \int_{0}^{t}\Bigl(\int_{0}^{u}e^{-\lambda (u-w)}\dd w\Bigr)\Bigl(\int_{0}^{u}e^{-\lambda (u-w)}\frac{(1-e^{-\lambda w})^2}{w^{3-2H}}\dd w\Bigr)\dd u\\
&\le \lambda^{-1}\int_{0}^{t}\int_{w}^{t}e^{-\lambda(u-w)}\dd u\frac{(1-e^{-\lambda w})^2}{w^{3-2H}}\dd w\\
&\le \lambda^{-2}\int_{0}^{t}\frac{(1-e^{-\lambda w})^2}{w^{3-2H}}\dd w.
\end{align*}
Finally, using a change of variables $r=\lambda w$, one has
\[
\int_{0}^{t}\frac{(1-e^{-\lambda w})^2}{w^{3-2H}}\dd w=\lambda^{2-2H}\int_{0}^{\lambda t}\frac{(1-e^{-r})^2}{r^{3-2H}}\dd r\le\lambda^{2-2H}\int_{0}^{\infty}\frac{(1-e^{-r})^2}{r^{3-2H}}\dd r,
\]
where using the condition $H\in(0,1/2)$ one has
\[
\int_{0}^{\infty}\frac{(1-e^{-r})^2}{r^{3-2H}}\dd r\le \int_{0}^{1}r^{2H-1}\dd r+\int_{0}^{\infty}\frac{1}{r^{3-2H}}\dd r<\infty.
\]
As a result, there exists $C_H\in(0,\infty)$ such that for all $\lambda\in(0,\infty)$ and $t\ge 0$ one has
\begin{equation}\label{eq:auxmoments2}
\mathcal{I}_{H,\lambda,t}^{2}=\int_{0}^{t}\Big|s^{\frac12-H}\int_{s}^{t}\frac{\tau^{H-\frac12}\bigl(e^{-\lambda(t-s)}-e^{-\lambda(t-\tau)}\bigr)}{(\tau-s)^{\frac32-H}}\dd \tau\Big|^2 \dd s\le C_H\lambda^{-2H}.
\end{equation}
Combining the auxiliary bounds~\eqref{eq:auxmoments0},~\eqref{eq:auxmoments1} and~\eqref{eq:auxmoments2}, one obtains~\eqref{eq:moment1DH-} and the proof of Lemma~\ref{lem:moment1DH-} is completed.
\end{proof}

\subsection{Temporal regularity}

\begin{lemma}\label{lem:temp1DH-}%[Temporal regularity for $H\in(0,\frac12)$]
For all $H\in(0,\frac12)$ and for all $\alpha\in(0,2H)$, there exists $C_{\alpha,H}\in(0,\infty)$ such that for all $\lambda\in(0,\infty)$, and all $t_1,t_2\ge 0$, one has
\begin{equation}\label{eq:temp1DH-}
\E[|\Z^{H,\lambda}(t_2)-\Z^{H,\lambda}(t_1)|^2]\le C_{\alpha,H}\lambda^{\alpha-2H}|t_2-t_1|^\alpha.
\end{equation}
\end{lemma}

\begin{rem}
Note that the inequality~\eqref{eq:temp1DH-} from Lemma~\ref{lem:temp1DH-} is proved only for $\alpha<2H$ (when $H\in(0,1/2)$), whereas the inequality~\eqref{eq:temp1DH+} from Lemma~\ref{lem:temp1DH+} is proved for $\alpha\le 2H$ (when $H\in(1/2,1)$).
\end{rem}

\begin{proof}[Proof of Lemma~\ref{lem:temp1DH-}]
Recall the decomposition~\eqref{eq:incrementsZHlambda} of $\Z^{H,\lambda}(t_2)-\Z^{H,\lambda}(t_1)$.

For the first term in~\eqref{eq:incrementsZHlambda}, using the inequality
\[
|1-e^{-z}|\le z^{\gamma},\qquad \forall~z\ge 0,\quad \forall~\gamma\in[0,1],
\]
and applying the moment bounds~\eqref{eq:moment1DH-} from Lemma~\ref{lem:moment1DH-}, one has
\begin{equation}\label{eq:aux1tempH-}
\E[\big|\bigl(e^{-\lambda(t_2-t_1)}-1\bigr)\Z^{H,\lambda}(t_1)\big|^2]\le C_H\lambda^{\alpha-2H}(t_2-t_1)^{\alpha}.
\end{equation}

For the second term in~\eqref{eq:incrementsZHlambda}, applying the It\^o isometry formula~\eqref{eq:Ito} and using the expression~\eqref{eq:KH-} for $\mathcal{K}_{H,T}^\star\phi (t)$ when $H\in(0,\frac12)$, one has
\begin{align*}
\E[|\int_{t_1}^{t_2}e^{-\lambda(t_2-s)}\dd\beta^H(s)|^2]&=\int_{0}^{t_2}|\mathcal{K}_{H,t_2}^\star\psi_{\lambda,t_1t_2}(s)|^2 \dd s\\
&\le C_H\int_{0}^{t_2}\frac{|\psi_{\lambda,t_1,t_2}(s)|^2}{(t_2-s)^{1-2H}}\dd s\\
&+C_H\int_{0}^{t_2}\Big|s^{\frac12-H}\int_{s}^{t_2}\frac{s^{H-\frac12}\psi_{\lambda,t_1,t_2}(s)-\tau^{H-\frac12}\psi_{\lambda,t_1,t_2}(\tau)}{(\tau-s)^{\frac32-H}}\dd \tau\Big|^2 \dd s.
\end{align*}
For the first term, for all $t_2\ge t_1\ge 0$, one has
\begin{align*}
\int_{0}^{t_2}\frac{|\psi_{\lambda,t_1,t_2}(s)|^2}{(t-s)^{1-2H}}\dd s&=\int_{t_1}^{t_2}\frac{e^{-2\lambda(t_2-s)}}{(t_2-s)^{1-2H}}\dd s=\int_{0}^{t_2-t_1}\frac{e^{-2\lambda s}}{s^{1-2H}}\dd s\\
&=\lambda^{-2H}\int_{0}^{\lambda(t_2-t_1)}\frac{e^{-2r}}{r^{1-2H}}\dd r,
\end{align*}
using a change of $r=\lambda s$. There exists $C_H\in(0,\infty)$ such that the following holds: for all $\alpha\in[0,1]$ and all $\theta\ge 0$, one has
\[
\int_{0}^{\theta}\frac{e^{-2r}}{r^{1-2H}}\dd r\le C_H\theta^{2H}\mathds{1}_{\theta\le 1}+\mathds{1}_{\theta\ge 1}\int_{0}^{\infty}\frac{e^{-2r}}{r^{1-2H}}\dd r\le C_H\min(\theta,1)\le C_H\theta^{\alpha}.
\]
Applying the inequality above with $\theta=\lambda(t_2-t_1)$, one obtains the following inequality: there exists $C_H\in(0,\infty)$ such that for all $\alpha\in[0,2H]$, $\lambda\in(0,\infty)$ and $t_2\ge t_1\ge 0$, one has
\begin{equation}\label{eq:auxtemp0}
\int_{0}^{t_2}\frac{|\psi_{\lambda,t_1,t_2}(s)|^2}{(t-s)^{1-2H}}\dd s\le C_H\lambda^{\alpha-2H}(t_2-t_1)^{\alpha}.
\end{equation}

For the second term, note that one has the decomposition
\[
s^{H-\frac12}\psi_{\lambda,t_1,t_2}(s)-\tau^{H-\frac12}\psi_{\lambda,t_1,t_2}(\tau)=\bigl(s^{H-\frac12}-\tau^{H-\frac12}\bigr)\psi_{\lambda,t_1,t_2}(s)+\tau^{H-\frac12}\bigl(\psi_{\lambda,t_1,t_2}(s)-\psi_{\lambda,t_1,t_2}(\tau)\bigr),
\]
therefore one obtains the upper bound
\[
\int_{0}^{t_2}\Big|s^{\frac12-H}\int_{s}^{t_2}\frac{s^{H-\frac12}\psi_{\lambda,t_1,t_2}(s)-\tau^{H-\frac12}\psi_{\lambda,t_1,t_2}(\tau)}{(\tau-s)^{\frac32-H}}\dd \tau\Big|^2 \dd s\le 2\bigl(\mathcal{J}_{H,\lambda,t_1,t_2}^1+\mathcal{J}_{H,\lambda,t_1,t_2}^2\bigr),
\]
where $\mathcal{J}_{H,\lambda,t_1,t_2}^1$ and $\mathcal{J}_{H,\lambda,t_1,t_2}^2$ are defined by
\begin{align*}
\mathcal{J}_{H,\lambda,t_1,t_2}^1&=\int_{0}^{t_2}\Big|s^{\frac12-H}\psi_{\lambda,t_1,t_2}(s)\int_{s}^{t_2}\frac{s^{H-\frac12}-\tau^{H-\frac12}}{(\tau-s)^{\frac32-H}}\dd \tau\Big|^2 \dd s\\
\mathcal{J}_{H,\lambda,t_1,t_2}^2&=\int_{0}^{t_2}\Big|s^{\frac12-H}\int_{s}^{t_2}\frac{\tau^{H-\frac12}\bigl(\psi_{\lambda,t_1,t_2}(s)-\psi_{\lambda,t_1,t_2}(\tau)\bigr)}{(\tau-s)^{\frac32-H}}\dd \tau\Big|^2 \dd s.
\end{align*}
To deal with the term $\mathcal{J}_{H,\lambda,t_1,t_2}^1$, applying the inequality~\eqref{eq:claim} (see the proof of Lemma~\ref{lem:moment1DH-}), one obtains
\begin{align*}
\int_{0}^{t_2}\Big|s^{\frac12-H}\psi_{\lambda,t_1,t_2}(s)\int_{s}^{t_2}\frac{s^{H-\frac12}-\tau^{H-\frac12}}{(\tau-s)^{\frac32-H}}\dd \tau\Big|^2 \dd s
&\le C_H\int_{0}^{t_2}s^{2H-1}\psi_{\lambda,t_1,t_2}(s)^2 \dd s\\
&\le C_H\int_{t_1}^{t_2}s^{2H-1}e^{-2\lambda(t_2-s)}\dd s\\
&\le C_H\lambda^{-2H}\int_{\lambda t_1}^{\lambda t_2}r^{2H-1}e^{-2(\lambda t_2-r)}\dd r,
\end{align*}
using a change of variables $r=\lambda s$.

We prove another auxiliary inequality: for all $\alpha\in[0,2H]$, there exists $C_{\alpha,H}\in(0,\infty)$ such that for all one has
\begin{equation}\label{eq:claim2}
\int_{\lambda t_1}^{\lambda t_2}r^{2H-1}e^{-2(\lambda t_2-r)}\dd r\le C_{\alpha,H}\lambda^\alpha (t_2-t_1)^{\alpha},\qquad \forall~t_2\ge t_1\ge 0.
\end{equation}
The inequality~\eqref{eq:claim2} is proved as follows. First, for all $t_2\ge t_1\ge 0$, one has
\begin{align*}
\int_{\lambda t_1}^{\lambda t_2}r^{2H-1}e^{-2(\lambda t_2-r)}\dd r&\le \int_{0}^{\lambda t_2}r^{2H-1}e^{-2(\lambda t_2-r)}\dd r\\
&\le \int_{0}^{1}r^{2H-1}e^{-2(\lambda t_2-r)}\dd r+\mathds{1}_{\lambda t_2\ge 1}\int_{1}^{\lambda t_2}r^{2H-1}e^{-2(\lambda t_2-r)}\dd r\\
&\le \int_{0}^{1}r^{2H-1}\dd r+\int_{0}^{\lambda t_2}e^{-2(\lambda t_2-r)}\dd r\\
&\le \frac{1}{2H}+\frac{1}{2}.
\end{align*}
Second, assuming that $\lambda(t_2-t_1)\le 1$, one has
\[
\int_{\lambda t_1}^{\lambda t_2}r^{2H-1}e^{-2(\lambda t_2-r)}\dd r\le \int_{\lambda t_1}^{\lambda t_2}r^{2H-1}\dd r=\frac{(\lambda t_2)^{2H}-(\lambda t_1)^{2H}}{2H}\le C_H\bigl(\lambda(t_2-t_1)\bigr)^{2H}.
\]
Combining the two upper bounds above, one obtains the following inequality: there exists $C_H\in(0,\infty)$ such that, for all $\lambda\in(0,\infty)$, $\alpha\in[0,2H]$ and $t_2\ge t_1\ge 0$, one has
\[
\int_{\lambda t_1}^{\lambda t_2}r^{2H-1}e^{-2(\lambda t_2-r)}\dd r\le C_H\min\bigl(1,\lambda(t_2-t_1)\bigr)^{2H}\le C_H\min\bigl(1,\lambda(t_2-t_1)\bigr)^{\alpha}\le C_H\lambda^\alpha (t_2-t_1)^\alpha,
\]
and the proof of the inequality~\eqref{eq:claim2} is completed.

Applying the inequality~\eqref{eq:claim2}, for the term $\mathcal{J}_{H,\lambda,t_1,t_2}^1$, one then obtains the following upper bound: there exists $C_H\in(0,\infty)$ such that for all $\alpha\in[0,2H]$, $\lambda\in(0,\infty)$ and $t_2\ge t_1\ge 0$ one has
\begin{equation}\label{eq:auxtemp1}
\mathcal{J}_{H,\lambda,t_1,t_2}^1=\int_{0}^{t_2}\Big|s^{\frac12-H}\psi_{\lambda,t_1,t_2}(s)\int_{s}^{t_2}\frac{s^{H-\frac12}-\tau^{H-\frac12}}{(\tau-s)^{\frac32-H}}\dd \tau\Big|^2 \dd s\le C_{\alpha,H}\lambda^{\alpha-2H} (t_2-t_1)^{\alpha}.
\end{equation}

It remains to deal with the error term $\mathcal{J}_{H,\lambda,t_1,t_2}^2$. First, one has the upper bound
\begin{align*}
\int_{0}^{t_2}\Big|s^{\frac12-H}\int_{s}^{t_2}\frac{\tau^{H-\frac12}\bigl(\psi_{\lambda,t_1,t_2}(s)-\psi_{\lambda,t_1,t_2}(\tau)\bigr)}{(\tau-s)^{\frac32-H}}\dd \tau\Big|^2 \dd s\le \int_{0}^{t_2}\Big|\int_{s}^{t_2}\frac{\big|\psi_{\lambda,t_1,t_2}(s)-\psi_{\lambda,t_1,t_2}(\tau)\big|}{(\tau-s)^{\frac32-H}}\dd \tau\Big|^2 \dd s.
\end{align*}

For all $\tau\in[s,t_2]$, one has
\begin{align*}
\psi_{\lambda,t_1,t_2}(s)-\psi_{\lambda,t_1,t_2}(\tau)&=e^{-\lambda(t_2-s)}\mathds{1}_{s\ge t_1}-e^{-\lambda(t_2-\tau)}\mathds{1}_{\tau\ge t_1}\\
&=e^{-\lambda(t_2-s)}\bigl(\mathds{1}_{s\ge t_1}-\mathds{1}_{\tau\ge t_1}\bigr)+\bigl(e^{-\lambda(t_2-s)}-e^{-\lambda(t_2-\tau)}\bigr)\mathds{1}_{\tau\ge t_1},
\end{align*}
and note that the following identity is satisfied:
\[
\mathds{1}_{\tau\ge s}\bigl(\mathds{1}_{s\ge t_1}-\mathds{1}_{\tau\ge t_1}\bigr)=-\mathds{1}_{\tau\ge t_1>s}.
\]
As a result, the error term $\mathcal{J}_{H,\lambda,t_1,t_2}^{2}$ can be treated as follows: one has
\begin{align*}
\mathcal{J}_{H,\lambda,t_1,t_2}^{2}&=\int_{0}^{t_2}\Big|\int_{s}^{t_2}\frac{\big|\psi_{\lambda,t_1,t_2}(s)-\psi_{\lambda,t_1,t_2}(\tau)\big|}{(\tau-s)^{\frac32-H}}\dd \tau\Big|^2 \dd s\\
&\le 2\int_{0}^{t_2}\Big|\int_{s}^{t_2}\frac{e^{-\lambda(t_2-s)}\mathds{1}_{\tau\ge t_1>s}}{(\tau-s)^{\frac32-H}}\dd \tau\Big|^2 \dd s+2\int_{0}^{t_2}\Big|\int_{s}^{t_2}\frac{\big|e^{-\lambda(t_2-s)}-e^{-\lambda(t_2-\tau)}\big|\mathds{1}_{\tau\ge t_1}}{(\tau-s)^{\frac32-H}}\dd \tau\Big|^2 \dd s\\
&\le 2\bigl(\mathcal{J}_{H,\lambda,t_1,t_2}^{2,1}+\mathcal{J}_{H,\lambda,t_1,t_2}^{2,2}+\mathcal{J}_{H,\lambda,t_1,t_2}^{2,3}\bigr),
\end{align*}
where $\mathcal{J}_{H,\lambda,t_1,t_2}^{2,1}$, $\mathcal{J}_{H,\lambda,t_1,t_2}^{2,2}$ and $\mathcal{J}_{H,\lambda,t_1,t_2}^{2,3}$ are defined by
\begin{align*}
\mathcal{J}_{H,\lambda,t_1,t_2}^{2,1}&=2\int_{0}^{t_1}\Big|\int_{t_1}^{t_2}\frac{e^{-\lambda(t_2-s)}}{(\tau-s)^{\frac32-H}}\dd \tau\Big|^2 \dd s\\
\mathcal{J}_{H,\lambda,t_1,t_2}^{2,2}&=2\int_{0}^{t_1}\Big|\int_{t_1}^{t_2}\frac{\big|e^{-\lambda(t_2-s)}-e^{-\lambda(t_2-\tau)}\big|}{(\tau-s)^{\frac32-H}}\dd \tau\Big|^2 \dd s\\
\mathcal{J}_{H,\lambda,t_1,t_2}^{2,3}&=2\int_{t_1}^{t_2}\Big|\int_{s}^{t_2}\frac{\big|e^{-\lambda(t_2-s)}-e^{-\lambda(t_2-\tau)}\big|}{(\tau-s)^{\frac32-H}}\dd \tau\Big|^2 \dd s.
\end{align*}
For the term $\mathcal{J}_{H,\lambda,t_1,t_2}^{2,1}$, one has
\begin{align*}
\int_{0}^{t_1}\Big|\int_{t_1}^{t_2}\frac{e^{-\lambda(t_2-s)}}{(\tau-s)^{\frac32-H}}\dd \tau\Big|^2 \dd s&=\int_{0}^{t_1}e^{-2\lambda(t_2-s)}\Big|\int_{t_1}^{t_2}\frac{1}{(\tau-s)^{\frac32-H}}\dd \tau\Big|^2 \dd s\\
&=\int_{0}^{t_1}e^{-2\lambda(t_2-s)}\Big|\frac{(t_2-s)^{H-\frac12}-(t_1-s)^{H-\frac12}}{H-\frac12}\Big|^2\dd s\\
&\le C_He^{-2\lambda(t_2-t_1)}\int_{0}^{t_1}e^{-2\lambda(t_1-s)}\bigl[(t_2-t_1+t_1-s)^{H-\frac12}-(t_1-s)^{H-\frac12}\bigr]^2 \dd s\\
&\le C_H\int_{0}^{t_1}e^{-2\lambda s}\bigl[(t_2-t_1+s)^{H-\frac12}-s^{H-\frac12}\bigr]^2 \dd s.
\end{align*}
Using a change of variable $r=\lambda s$, one has
\begin{align*}
\int_{0}^{t_1}e^{-2\lambda s}\bigl[(t_2-t_1+s)^{H-\frac12}-s^{H-\frac12}\bigr]^2 \dd s&=\lambda^{-2H}\int_{0}^{\lambda t_1}e^{-2r}\bigl[(\lambda(t_2-t_1)+r)^{H-\frac12}-r^{H-\frac12}\bigr]^2\dd r.
\end{align*}
Combining the inequalities
\[
\big|(\theta+r)^{H-\frac12}-r^{H-\frac12}\big|\le 2r^{H-\frac12},\quad \forall~\theta,r>0,
\]
and
\[
\big|(\theta+r)^{H-\frac12}-r^{H-\frac12}\big|\le C_H\theta r^{H-\frac32},\quad \forall~\theta,r>0,
\]
for all $\alpha\in[0,1]$ one obtains
\[
\big|(\theta+r)^{H-\frac12}-r^{H-\frac12}\big|\le C_{\alpha,H}\frac{\theta^{\alpha}}{r^{\alpha+\frac12-H}}.
\]
As a result, if the condition $\alpha<2H$ is satisfied, one obtains
\begin{align*}
\int_{0}^{\lambda t_1}e^{-2r}\bigl[(\lambda(t_2-t_1)+r)^{H-\frac12}-r^{H-\frac12}\bigr]^2\dd r&\le C_{\alpha,H}\bigl(\lambda(t_2-t_1)\bigr)^{\alpha}\int_{0}^{\lambda t_1}\frac{1}{r^{\alpha+1-2H}}\dd r\\
&\le C_{\alpha,H}\int_{0}^{\infty}\frac{1}{r^{\alpha+1-2H}}\dd r \lambda^\alpha (t_2-t_1)^\alpha.
\end{align*}
Therefore, for all $\alpha\in[0,2H)$, there exists $C_{\alpha,H}\in(0,\infty)$ such that for all $\lambda\in(0,\infty)$ and $t_2\ge t_1\ge 0$, one has
\begin{equation}\label{eq:auxtemp21}
\mathcal{J}_{H,\lambda,t_1,t_2}^{2,1}\le C_{\alpha,H}\lambda^{\alpha-2H} (t_2-t_1)^\alpha.
\end{equation}

%under the condition $\alpha<2H$, to ensure that $\int_{0}^{\infty}\frac{1}{r^{\alpha+1-2H}}\dd r<\infty$.

For the term $\mathcal{J}_{H,\lambda,t_1,t_2}^{2,2}$, applying the Cauchy--Schwarz inequality, one has
\begin{align*}
\mathcal{J}_{H,\lambda,t_1,t_2}^{2,2}&=\int_{0}^{t_1}\Big|\int_{t_1}^{t_2}\frac{\big|e^{-\lambda(t_2-s)}-e^{-\lambda(t_2-\tau)}\big|}{(\tau-s)^{\frac32-H}}\dd \tau\Big|^2 \dd s\\
&=\int_{0}^{t_1}\Big|\int_{t_1}^{t_2}\frac{e^{-\lambda(t_2-\tau)}\bigl(1-e^{-\lambda(\tau-s)}\bigr)}{(\tau-s)^{\frac32-H}}\dd \tau\Big|^2 \dd s\\
&\le \int_{0}^{t_1}\Bigl(\int_{t_1}^{t_2}e^{-\lambda(t_2-\tau)}\dd \tau\Bigr)\Bigl(\int_{t_1}^{t_2}e^{-\lambda(t_2-\tau)}\frac{\bigl(1-e^{-\lambda(\tau-s)}\bigr)^2}{(\tau-s)^{3-2H}}\dd \tau\Bigr)\dd s\\
&\le \frac{1-e^{-\lambda(t_2-t_1)}}{\lambda}\int_{0}^{t_1}\int_{t_1}^{t_2}e^{-\lambda(t_2-\tau)}\frac{\bigl(1-e^{-\lambda(\tau-s)}\bigr)^2}{(\tau-s)^{3-2H}}\dd \tau \dd s.
\end{align*}
Note that for all $\alpha\in[0,1]$, one has
\[
1-e^{-\lambda(t_2-t_1)}\le \bigl(\lambda(t_2-t_1)\bigr)^\alpha.
\]
Using changes of variables $u=\lambda \tau$ and $v=\lambda s$, then $u'=\lambda t_2-u$ and $v'=\lambda t_2-v$, one obtains
\begin{align*}
\int_{0}^{t_1}\int_{t_1}^{t_2}e^{-\lambda(t_2-\tau)}\frac{\bigl(1-e^{-\lambda(\tau-s)}\bigr)^2}{(\tau-s)^{3-2H}}\dd \tau \dd s&=\lambda^{1-2H}\int_{0}^{\lambda t_1}\int_{\lambda t_1}^{\lambda t_2}e^{-(\lambda t_2-u)}\frac{\bigl(1-e^{-(u-v)}\bigr)^2}{(u-v)^{3-2H}}\dd u \dd v\\
&=\lambda^{1-2H}\int_{\lambda(t_2-t_1)}^{\lambda t_2}\int_{0}^{\lambda (t_2-t_1)}e^{-u'}\frac{\bigl(1-e^{-(v'-u')}\bigr)^2}{(v'-u')^{3-2H}}\dd u' \dd v'\\
&\le\lambda^{1-2H}\int_{0}^{\lambda (t_2-t_1)}e^{-u'}\int_{\lambda(t_2-t_1)}^{\lambda t_2}\frac{\bigl(1-e^{-(v'-u')}\bigr)^2}{(v'-u')^{3-2H}}\dd v' \dd u'\\
&\le \lambda^{1-2H}\int_{0}^{\infty}e^{-u}\dd u \int_{0}^{\infty}\frac{\bigl(1-e^{-w}\bigr)^2}{w^{3-2H}}\dd w\\
&\le C_H\lambda^{1-2H}.
\end{align*}
Therefore, for all $\alpha\in[0,2H)$, there exists $C_{\alpha,H}\in(0,\infty)$ such that for all $\lambda\in(0,\infty)$ and $t_2\ge t_1\ge 0$, one has
\begin{equation}\label{eq:auxtemp22}
\mathcal{J}_{H,\lambda,t_1,t_2}^{2,2}\le C_{\alpha,H}\lambda^{\alpha-2H}(t_2-t_1)^{\alpha}.
\end{equation}
For the term $\mathcal{J}_{H,\lambda,t_1,t_2}^{2,3}$, one has
\begin{align*}
\int_{t_1}^{t_2}\Big|\int_{s}^{t_2}\frac{\big|e^{-\lambda(t_2-s)}-e^{-\lambda(t_2-\tau)}\big|}{(\tau-s)^{\frac32-H}}\dd \tau\Big|^2 \dd s&=\int_{t_1}^{t_2}\Big|\int_{s}^{t_2}\frac{e^{-\lambda(t_2-\tau)}\bigl(1-e^{-\lambda(\tau-s)}\bigr)}{(\tau-s)^{\frac32-H}}\dd \tau\Big|^2 \dd s\\
&\le \int_{t_1}^{t_2}\Bigl(\int_{s}^{t_2}e^{-\lambda(t_2-\tau)}\dd \tau\Bigr)\Bigl(\int_{s}^{t_2}e^{-\lambda(t_2-\tau)}\frac{\bigl(1-e^{-\lambda(\tau-s)}\bigr)^2}{(\tau-s)^{3-2H}}\dd \tau\Bigr)\dd s.
\end{align*}
Note that for all $t_2\ge s\ge 0$ and $\lambda\in(0,\infty)$, one has
\[
\int_{s}^{t_2}e^{-\lambda(t_2-\tau)}\dd \tau=\lambda^{-1}\bigl(1-e^{-\lambda(t_2-s)}\bigr)\le \lambda^{-1}.
\]
%On the other hand, one has
%\[
%\int_{s}^{t_2}e^{-\lambda(t_2-\tau)}\frac{\bigl(1-e^{-\lambda(\tau-s)}\bigr)^2}{(\tau-s)^{3-2H}}\dd \tau\le \lambda^2 \int_{s}^{t_2}e^{-\lambda(t_2-\tau)}\frac{1}{(\tau-s)^{1-2H}}\dd \tau.
%\]
Thus, using a change of variables $u=\lambda\tau$ and $v=\lambda s$, one obtains
\begin{align*}
\int_{t_1}^{t_2}\Big|\int_{s}^{t_2}\frac{\big|e^{-\lambda(t_2-s)}-e^{-\lambda(t_2-\tau)}\big|}{(\tau-s)^{\frac32-H}}\dd \tau\Big|^2 \dd s
&\le \lambda^{-1}\int_{t_1}^{t_2}\int_{s}^{t_2}e^{-\lambda(t_2-\tau)}\frac{\bigl(1-e^{-\lambda(\tau-s)}\bigr)^2}{(\tau-s)^{3-2H}}\dd \tau \dd s\\
&\le \lambda^{-2H}\int_{\lambda t_1}^{\lambda t_2}\int_{v}^{\lambda t_2}e^{-(\lambda t_2-u)}\frac{\bigl(1-e^{-(u-v)}\bigr)^2}{(u-v)^{3-2H}}\dd u \dd v\\
&\le \lambda^{-2H}\int_{\lambda t_1}^{\lambda t_2}\int_{\lambda t_1}^{u}e^{-(\lambda t_2-u)}\frac{\bigl(1-e^{-\lambda(u-v)}\bigr)^2}{(u-v)^{3-2H}}\dd v \dd u\\
&\le \lambda^{-2H}\int_{\lambda t_1}^{\lambda t_2}e^{-(\lambda t_2-u)}\dd u \int_{0}^{\infty}\frac{\bigl(1-e^{-w}\bigr)^2}{w^{3-2H}}\dd w.
%&\le C_H\lambda^{-2H}\bigl(\lambda(t_2-t_1)\bigr)^{\alpha},
\end{align*}
Note that for all $\alpha\in[0,2H]$ one has
\[
\int_{\lambda t_1}^{\lambda t_2}e^{-(\lambda t_2-u)}\dd u\le \bigl(\lambda(t_2-t_1)\bigr)^{\alpha}.
\]

Therefore, for all $\alpha\in[0,2H)$, there exists $C_{\alpha,H}\in(0,\infty)$ such that for all $\lambda\in(0,\infty)$ and $t_2\ge t_1\ge 0$, one has
\begin{equation}\label{eq:auxtemp23}
\mathcal{J}_{H,\lambda,t_1,t_2}^{2,3}\le C_{\alpha,H}\lambda^{\alpha-2H}(t_2-t_1)^{\alpha}.
\end{equation}

Recalling that
\[
\mathcal{J}_{H,\lambda,t_1,t_2}^{2}\le 2\bigl(\mathcal{J}_{H,\lambda,t_1,t_2}^{2,1}+\mathcal{J}_{H,\lambda,t_1,t_2}^{2,2}+\mathcal{J}_{H,\lambda,t_1,t_2}^{2,3}\bigr)
\]
and combining the inequalities~\eqref{eq:auxtemp21},~\eqref{eq:auxtemp22} and~\eqref{eq:auxtemp23} yields the following inequality: for all $\alpha\in[0,2H)$, there exists $C_{\alpha,H}\in(0,\infty)$ such that for all $\lambda\in(0,\infty)$ and $t_2\ge t_1\ge 0$, one has
\begin{equation}\label{eq:auxtemp2}
\mathcal{J}_{H,\lambda,t_1,t_2}^{2}=\int_{0}^{t_2}\Big|\int_{s}^{t_2}\frac{\big|\psi_{\lambda,t_1,t_2}(s)-\psi_{\lambda,t_1,t_2}(\tau)\big|}{(\tau-s)^{\frac32-H}}\dd \tau\Big|^2 \dd s\le C_{\alpha,H}\lambda^{\alpha-2H} (t_2-t_1)^{\alpha}.
\end{equation}

Finally, combining the inequalities~\eqref{eq:auxtemp0},~\eqref{eq:auxtemp1} and~\eqref{eq:auxtemp2}, one obtains the following inequality: for all $\alpha\in[0,2H)$, there exists $C_{\alpha,H}\in(0,\infty)$ such that for all $\lambda\in(0,\infty)$ and $t_2\ge t_1\ge 0$, one has
\begin{equation}\label{eq:aux2tempH-}
\E[|\Z^{H,\lambda}(t_2)-\Z^{H,\lambda}(t_1)|^2]\le C_{\alpha,H}\lambda^{\alpha-2H} (t_2-t_1)^{\alpha}
\end{equation}

Gathering the estimates~\eqref{eq:aux1tempH-} and~\eqref{eq:aux2tempH-} and recalling the decomposition~\eqref{eq:incrementsZHlambda}, one obtains~\eqref{eq:temp1DH-} and the proof of Lemma~\ref{lem:temp1DH-} is completed.
\end{proof}

\section{Appendix: error estimates for the numerical approximation of the fractional Ornstein--Uhlenbeck process}\label{app:D}

\subsection{Case $H\in(1/2,1)$}

\begin{lemma}\label{lem:error1DH+}
For all $H\in(1/2,1)$, there exists $C_H\in(0,\infty)$ such that for  all $\lambda\in(0,\infty)$ and all $\Delta t\in(0,1)$ one has
\begin{equation}\label{eq:error1DH+}
\underset{n\in\N}\sup~\E[|\Z^{H,\lambda}(t_n)-\Z_n^{H,\lambda}|^2]\le C_H\lambda^{-2H}\min\bigl(1,\lambda \Delta t\bigr)^2.
\end{equation}
\end{lemma}

\begin{proof}[Proof of Lemma~\ref{lem:error1DH+}]
Recall the expression~\eqref{eq:errorZnHlambda} of the error $\Z^{H,\lambda}(t_n)-\Z_n^{H,\lambda}$. Applying the It\^o isometry formula~\eqref{eq:Ito} and its formulation~\eqref{eq:Ito+} when $H\in(1/2,1)$, for all $n\ge 0$ one has
\begin{align*}
\E[|\Z^{H,\lambda}(t_n)-\Z_n^{H,\lambda}|^2]&=\int_{0}^{t_n}|\mathcal{K}_{H,t_n}^\star\varepsilon_{\lambda,t_n}^{\Delta t}(t)|^2 \dd t\\
&=H(2H-1)\int_{0}^{t_n}\int_{0}^{t_n}\varepsilon_{\lambda,t_n}^{\Delta t}(s_1)\varepsilon_{\lambda,t_n}^{\Delta t}(s_2)|s_2-s_1|^{2H-2}\dd s_1\dd s_2.
\end{align*}
Note that for all $s\in[0,t_n]$, one has
\[
\varepsilon_{\lambda,t_n}^{\Delta t}(s)=e^{-\lambda(t_n-s)}\bigl(1-e^{-\lambda(s-t_{\ell}(s))}\bigr),
\]
and that one has
\[
\bigl(1-e^{-\lambda(s-t_{\ell}(s))}\bigr)\le \min\bigl(1,\lambda\Delta t\bigr).
\]
As a result, one obtains the inequality
\[
\E[|X^{H,\lambda}(t_n)-x_n^{H,\lambda}|^2]\le \min\bigl(1,\lambda\Delta t\bigr)^2H(2H-1)\int_{0}^{t}\int_{0}^{t}e^{-\lambda(t-s_1)}e^{-\lambda(t-s_2)}|s_2-s_1|^{2H-2}\dd s_1 \dd s_2.
\]
Applying the upper bounds obtained in the proof of Lemma~\ref{lem:moment1DH+}, the inequality~\eqref{eq:error1DH+} holds, and the proof of Lemma~\ref{lem:error1DH+} is completed.
\end{proof}

Note that combining Lemma~\ref{lem:moment1DH+} and~\ref{lem:error1DH+}, one obtains the following moment bounds for the discrete-time approximation $\bigl(\Z_n^{H,\lambda}\bigr)_{n\ge 0}$ of $\bigl(\Z^{H,\lambda}(t_n)\bigr)_{n\ge 0}$.
\begin{lemma}\label{lem:momentnum1DH+}
For all $H\in(1/2,1)$, there exists $C_H\in(0,\infty)$ such that for  all $\lambda\in(0,\infty)$ one has
\begin{equation}\label{eq:momentnum1DH+}
\underset{\Delta t\in(0,1)}\sup~\underset{n\ge 0}\sup~\E[|\Z_n^{H,\lambda}|^2]\le C_H\lambda^{-2H}.
\end{equation}
\end{lemma}

\subsection{Case $H\in(0,1/2)$}

\begin{lemma}\label{lem:error1DH-}
For all $H\in(0,\frac12)$ and for all $\alpha\in[0,2H)$, there exists $C_{\alpha,H}\in(0,\infty)$ such that for all $\lambda\in(0,\infty)$ and all $\Delta t\in(0,1)$ one has
\begin{equation}\label{eq:error1DH-}
\underset{n\in\N}\sup~\E[|\Z^{H,\lambda}(t_n)-\Z_n^{H,\lambda}|^2]\le C_H\lambda^{-2H}(\lambda \Delta t)^{2\alpha}.
\end{equation}
\end{lemma}

\begin{proof}[Proof of Lemma~\ref{lem:error1DH-}]
Recall the expression~\eqref{eq:errorZnHlambda} of the error $\Z^{H,\lambda}(t_n)-\Z_n^{H,\lambda}$. Applying the It\^o isometry formula~\eqref{eq:Ito} and using the expression~\eqref{eq:KH-} for $\mathcal{K}_{H,T}^\star\phi (t)$ when $H\in(0,\frac12)$, one has
\begin{align*}
\E[|\Z^{H,\lambda}(t_n)-\Z_n^{H,\lambda}|^2]&=\int_{0}^{t_n}|\mathcal{K}_{H,t_n}^\star\varepsilon_{\lambda,t_n}^{\Delta t}(t)|^2 \dd t\\
&\le 2c_H^2\int_{0}^{t_n}\frac{|\varepsilon_{\lambda,t_n}^{\Delta t}(s)|^2}{(t_n-s)^{1-2H}}\dd s+2c_H^2\int_{0}^{t_n}\Big|s^{\frac12-H}\int_{s}^{t_n}\frac{s^{H-\frac12}\varepsilon_{\lambda,t_n}^{\Delta t}(s)-\tau^{H-\frac12}\varepsilon_{\lambda,t_n}^{\Delta t}(\tau)}{(\tau-s)^{\frac32-H}}\dd \tau\Big|^2 \dd s.
\end{align*}

For the first term, note that one has
\begin{equation}\label{eq:auxepsilon}
|\varepsilon_{\lambda,t_n}^{\Delta t}(s)|\le \min\bigl(1,\lambda\Delta t\bigr)e^{-\lambda(t_n-s)},\quad \forall~s\in[0,t_n],
\end{equation}
and as a result one obtains
\[
\int_{0}^{t_n}\frac{|\varepsilon_{\lambda,t_n}^{\Delta t}(s)|^2}{(t_n-s)^{1-2H}}\dd s\le \min\bigl(1,\lambda\Delta t\bigr)^{2}\int_{0}^{t_n}\frac{e^{-2\lambda(t_n-s)}}{(t_n-s)^{1-2H}}\dd s.
\]
Proceeding as in the proof of Lemma~\ref{lem:moment1DH-}, one obtains the following inequality: there exists $C_H\in(0,\infty)$ such that for all $\lambda\in(0,\infty)$, all $\Delta t\in(0,1)$ and all $n\ge 0$ one has
\begin{equation}\label{eq:auxerror0}
\int_{0}^{t_n}\frac{|\varepsilon_{\lambda,t_n}^{\Delta t}(s)|^2}{(t_n-s)^{1-2H}}\dd s\le C_H\lambda^{-2H} \min\bigl(1,\lambda\Delta t\bigr)^{2}.
\end{equation}

For the second term, note that one has the decomposition
\[
s^{H-\frac12}\varepsilon_{\lambda,t_n}^{\Delta t}(s)-\tau^{H-\frac12}\varepsilon_{\lambda,t_n}^{\Delta t}(\tau)=\bigl(s^{H-\frac12}-\tau^{H-\frac12}\bigr)\varepsilon_{\lambda,t_n}^{\Delta t}(s)+\tau^{H-\frac12}\bigl(\varepsilon_{\lambda,t_n}^{\Delta t}(s)-\varepsilon_{\lambda,t_n}^{\Delta t}(\tau)\bigr),
\]
therefore one obtains the upper bound
\[
\int_{0}^{t_n}\Big|s^{\frac12-H}\int_{s}^{t_n}\frac{s^{H-\frac12}\varepsilon_{\lambda,t_n}^{\Delta t}(s)-\tau^{H-\frac12}\varepsilon_{\lambda,t_n}^{\Delta t}(\tau)}{(\tau-s)^{\frac32-H}}\dd \tau\Big|^2 \dd s\le 2\bigl(\mathcal{K}_{H,\lambda,\Delta t,n}^1+\mathcal{K}_{H,\lambda,\Delta t,n}^2\bigr),
\]
where $\mathcal{K}_{H,\lambda,\Delta t,n}^1$ and $\mathcal{K}_{H,\lambda,\Delta t,n}^2$ are defined as
\begin{align*}
\mathcal{K}_{H,\lambda,\Delta t,n}^1&=\int_{0}^{t_n}\Big|s^{\frac12-H}\varepsilon_{\lambda,t_n}^{\Delta t}(s)\int_{s}^{t_n}\frac{s^{H-\frac12}-\tau^{H-\frac12}}{(\tau-s)^{\frac32-H}}\dd \tau\Big|^2 \dd s\\
\mathcal{K}_{H,\lambda,\Delta t,n}^2&=\int_{0}^{t_n}\Big|s^{\frac12-H}\int_{s}^{t_n}\frac{\tau^{H-\frac12}\bigl(\varepsilon_{\lambda,t_n}^{\Delta t}(s)-\varepsilon_{\lambda,t_n}^{\Delta t}(\tau)\bigr)}{(\tau-s)^{\frac32-H}}\dd \tau\Big|^2 \dd s.
\end{align*}

To deal with the term $\mathcal{K}_{H,\lambda,\Delta t,n}^1$, owing to the inequality~\eqref{eq:auxepsilon}, one has
\begin{align*}
\int_{0}^{t_n}\Big|s^{\frac12-H}\varepsilon_{\lambda,t_n}^{\Delta t}(s)\int_{s}^{t_n}\frac{s^{H-\frac12}-\tau^{H-\frac12}}{(\tau-s)^{\frac32-H}}\dd \tau\Big|^2 \dd s
&\le \min\bigl(1,\lambda\Delta t\bigr)^2\int_{0}^{t_n}\Big|s^{\frac12-H}e^{-2\lambda(t_n-s)}\int_{s}^{t_n}\frac{s^{H-\frac12}-\tau^{H-\frac12}}{(\tau-s)^{\frac32-H}}\dd \tau\Big|^2 \dd s\\
&\le \min\bigl(1,\lambda\Delta t\bigr)^2\mathcal{I}_{H,\lambda,t_n}^{1},
\end{align*}
where $\mathcal{I}_{H,\lambda,t_n}^{1}$ is defined in the proof of Lemma~\ref{lem:moment1DH-}. Applying the inequality~\eqref{eq:auxmoments1} then yields the following inequality: 
there exists $C_H\in(0,\infty)$ such that for all $\lambda\in(0,\infty)$, all $\Delta t\in(0,1)$ and all $n\ge 0$ one has
\begin{equation}\label{eq:auxerror1}
\mathcal{K}_{H,\lambda,\Delta t,n}^1\le C_H \lambda^{-2H} \min\bigl(1,\lambda\Delta t\bigr)^{2}.
\end{equation}
It remains to deal with the term $\mathcal{K}_{H,\lambda,\Delta t,n}^2$. Owing to the decomposition
\begin{align*}
\varepsilon_{\lambda,t_n}^{\Delta t}(s)-\varepsilon_{\lambda,t_n}^{\Delta t}(\tau)&=e^{-\lambda(t_n-s)}\bigl(1-e^{-\lambda(s-t_{\ell(s)})}\bigr)-e^{-\lambda(t_n-\tau)}\bigl(1-e^{-\lambda(\tau-t_{\ell(\tau)})}\bigr)\\
&=\bigl(e^{-\lambda(t_n-s)}-e^{-\lambda(t_n-\tau)}\bigr)\bigl(1-e^{-\lambda(\tau-t_{\ell(\tau)})}\bigr)\\
&\quad+e^{-\lambda(t_n-s)}\bigl(e^{-\lambda(\tau-t_{\ell(\tau)})}-e^{-\lambda(s-t_{\ell(s)})}\bigr),
\end{align*}
one has
\[
\mathcal{K}_{H,\lambda,\Delta t,n}^2\le 2\bigl( \mathcal{K}_{H,\lambda,\Delta t,n}^{2,1}+\mathcal{K}_{H,\lambda,\Delta t,n}^{2,2}\bigr),
\]
where $\mathcal{K}_{H,\lambda,\Delta t,n}^{2,1}$ and $\mathcal{K}_{H,\lambda,\Delta t,n}^{2,2}$ are defined as
\begin{align*}
\mathcal{K}_{H,\lambda,\Delta t,n}^{2,1}&=\int_{0}^{t_n}\Big|\int_{s}^{t_n}\frac{\bigl(e^{-\lambda(t_n-s)}-e^{-\lambda(t_n-\tau)}\bigr)\bigl(1-e^{-\lambda(\tau-t_{\ell(\tau)})}\bigr)}{(\tau-s)^{\frac32-H}}\dd \tau\Big|^2 \dd s\\
\mathcal{K}_{H,\lambda,\Delta t,n}^{2,2}&=\int_{0}^{t_n}\Big|\int_{s}^{t_n}\frac{e^{-\lambda(t_n-s)}\bigl(e^{-\lambda(\tau-t_{\ell(\tau)})}-e^{-\lambda(s-t_{\ell(s)})}\bigr)}{(\tau-s)^{\frac32-H}}\dd \tau\Big|^2 \dd s.
\end{align*}
For the term $\mathcal{K}_{H,\lambda,\Delta t,n}^{2,1}$, using the upper bound
\[
|1-e^{-\lambda(\tau-t_{\ell(\tau)})}|\le \min\bigl(1,\lambda\Delta t\bigr),
\]
one obtains the upper bound
\[
\mathcal{K}_{H,\lambda,\Delta t,n}^{2,1}\le \min\bigl(1,\lambda\Delta t\bigr)^2 \int_{0}^{t_n}\Big|\int_{s}^{t_n}\frac{\bigl(e^{-\lambda(t_n-s)}-e^{-\lambda(t_n-\tau)}\bigr)}{(\tau-s)^{\frac32-H}}\dd \tau\Big|^2 \dd s.
\]
Applying the inequality~\eqref{eq:auxmoments2} on the term $\mathcal{I}_{H,\lambda,t_n}^{2}$ from the proof of Lemma~\ref{lem:moment1DH-} then yields the following inequality: 
there exists $C_H\in(0,\infty)$ such that for all $\lambda\in(0,\infty)$, all $\Delta t\in(0,1)$ and all $n\ge 0$ one has
\begin{equation}\label{eq:auxerror21}
\mathcal{K}_{H,\lambda,\Delta t,n}^{2,1}\le C_H \lambda^{-2H} \min\bigl(1,\lambda\Delta t\bigr)^{2}.
\end{equation}
The treatment of the error term $\mathcal{K}_{H,\lambda,\Delta t,n}^{2,2}$ is more delicate.

One has the inequalities
\begin{equation}\label{eq:claim_error_1}
\big|e^{-\lambda(\tau-t_{\ell(\tau)})}-e^{-\lambda(s-t_{\ell(s)})}\big|\le \min\bigl(1,\lambda\Delta t)
\end{equation}
and
\begin{equation}\label{eq:claim_error_2}
\big|e^{-\lambda(\tau-t_{\ell(\tau)})}-e^{-\lambda(s-t_{\ell(s)})}\big|\le \min\bigl(1,\lambda(\tau-s)\bigr)+\mathds{1}_{\tau\in[t_{\ell(s)+1},t_{\ell(s)+2}[}\min\bigl(1,\lambda\Delta t)
\end{equation}
are satisfied. To prove the inequality~\eqref{eq:claim_error_1}, it suffices to observe that $\tau-t_{\ell(\tau)}\le \Delta t$ and $s-t_{\ell(s)}\le \Delta t$. To prove the inequality~\eqref{eq:claim_error_2}, three cases are considered:
\begin{itemize}
\item if $\tau>t_{\ell(s)+2}$, one has $\tau-s\ge \Delta t$, therefore
\[
\big|e^{-\lambda(\tau-t_{\ell(\tau)})}-e^{-\lambda(s-t_{\ell(s)})}\big|\le \min\bigl(1,\lambda\Delta t)\le \min\bigl(1,\lambda(\tau-s)\bigr);
\]
\item if $\tau<t_{\ell(s)+1}$, one has $t_{\ell(\tau)}=t_{\ell(s)}$, thus
\[
\big|e^{-\lambda(\tau-t_{\ell(\tau)})}-e^{-\lambda(s-t_{\ell(s)})}\big|=\bigl(1-e^{-\lambda(\tau-s)}\bigr)e^{-\lambda(s-t_{\ell(s)})}\le \min\bigl(1,\lambda(\tau-s)\bigr);
\]
\item if $\tau\in[t_{\ell(s)+1},t_{\ell(s)+2}[$, one has $t_{\ell(\tau)}=t_{\ell(s)+1}=t_{\ell(s)}+\Delta t$
\begin{align*}
\big|e^{-\lambda(\tau-t_{\ell(\tau)})}-e^{-\lambda(s-t_{\ell(s)})}\big|&\le \big|e^{-\lambda(\tau-t_{\ell(s)}-\Delta t)}-e^{-\lambda(s-t_{\ell(s)})}\big|\\
&\le \big|e^{-\lambda(\tau-t_{\ell(s)}-\Delta t)}-e^{-\lambda(\tau-t_{\ell(s)})}\big|+\big|e^{-\lambda(\tau-t_{\ell(s)})}-e^{-\lambda(s-t_{\ell(s)})}|\\
&\le \big|1-e^{-\lambda\Delta t}\big|+\big|1-e^{-\lambda(\tau-s)}\big|\\
&\le \min\bigl(1,\lambda\Delta t)+\min\bigl(1,\lambda(\tau-s)\bigr).
\end{align*}
\end{itemize}
Gathering the inequalities obtained in the three cases yields the inequality~\eqref{eq:claim_error_2}.

Let $\alpha\in[0,2H)$. Combining the inequalities~\eqref{eq:claim_error_1} and~\eqref{eq:claim_error_2}, one obtains
\begin{align*}
\big|e^{-\lambda(\tau-t_{\ell(\tau)})}-e^{-\lambda(s-t_{\ell(s)})}\big|&\le \big|e^{-\lambda(\tau-t_{\ell(\tau)})}-e^{-\lambda(s-t_{\ell(s)})}\big|^{1-\alpha}\big|e^{-\lambda(\tau-t_{\ell(\tau)})}-e^{-\lambda(s-t_{\ell(s)})}\big|^{\alpha}\\
&\le \Bigl(\min\bigl(1,\lambda(\tau-s)\bigr)+\mathds{1}_{\tau\in[t_{\ell(s)+1},t_{\ell(s)+2}[}\min\bigl(1,\lambda\Delta t)\Bigr)^{1-\alpha}\Bigl(\min\bigl(1,\lambda\Delta t)\Bigr)^{\alpha}\\
&\le \min\bigl(1,\lambda(\tau-s)\bigr)^{1-\alpha}\min\bigl(1,\lambda\Delta t)^{\alpha}+\mathds{1}_{\tau\in[t_{\ell(s)+1},t_{\ell(s)+2}[}\min\bigl(1,\lambda\Delta t).
\end{align*}

As a result, one obtains the upper bound
\begin{align*}
\Big|\int_{s}^{t_n}\frac{\bigl(e^{-\lambda(\tau-t_{\ell(\tau)})}-e^{-\lambda(s-t_{\ell(s)})}\bigr)}{(\tau-s)^{\frac32-H}}\dd \tau\Big|
&\le \bigl(\lambda\Delta t\bigr)^{\alpha}\int_{s}^{t_n}\frac{\min\bigl(1,\lambda(\tau-s)\bigr)^{1-\alpha}}{(\tau-s)^{\frac32-H}}\dd \tau\\
&+\min\bigl(1,\lambda\Delta t\bigr)\int_{t_{\ell(s)+1}}^{t_{\ell(s)+2}}\frac{1}{(\tau-s)^{\frac32-H}}\dd \tau.
\end{align*}
As a result, one obtains
\[
\mathcal{K}_{H,\lambda,\Delta t,n}^{2,2}\le 2\Bigl(\mathcal{K}_{H,\lambda,\Delta t,n}^{2,2,1}+\mathcal{K}_{H,\lambda,\Delta t,n}^{2,2,2}\Bigr),
\]
where $\mathcal{K}_{H,\lambda,\Delta t,n}^{2,2}$ and $\mathcal{K}_{H,\lambda,\Delta t,n}^{2,2}$ are defined by
\begin{align*}
\mathcal{K}_{H,\lambda,\Delta t,n}^{2,2,1}&=\bigl(\lambda\Delta t\bigr)^{2\alpha}\int_{0}^{t_n}e^{-2\lambda(t_n-s)}
\Big|\int_{s}^{t_n}\frac{\min\bigl(1,\lambda(\tau-s)\bigr)^{1-\alpha}}{(\tau-s)^{\frac32-H}}\dd \tau\Big|^2 \dd s\\
\mathcal{K}_{H,\lambda,\Delta t,n}^{2,2,2}&=\min\bigl(1,\lambda\Delta t\bigr)^2\int_{0}^{t_n}e^{-2\lambda(t_n-s)}\Big|\int_{t_{\ell(s)+1}}^{t_{\ell(s)+2}}\frac{1}{(\tau-s)^{\frac32-H}}\dd \tau\Big|^2 \dd s
\end{align*}

For the term $\mathcal{K}_{H,\lambda,\Delta t,n}^{2,2,1}$, one has
\begin{align*}
\mathcal{K}_{H,\lambda,\Delta t,n}^{2,2,1}&\le \bigl(\lambda\Delta t\bigr)^{2\alpha}\int_{0}^{t_n}e^{-2\lambda(t_n-s)} \Bigl(\int_{0}^{\infty}\frac{\min(1,\lambda \tau')^{1-\alpha}}{(\tau')^{\frac32-H}}\dd \tau'\Bigr)^2 \dd s\\
&\le \lambda^{-2H}\bigl(\lambda\Delta t\bigr)^{2\alpha}\int_{0}^{\infty}e^{-2r}\dd r \Bigl(\int_{0}^{\infty}\frac{\min(1,r)^{1-\alpha}}{r^{\frac32-H}}\dd r\Bigr)^2,
\end{align*}
after changes of variables. Assuming that $\alpha<2H$ with $H\in(0,1/2)$ ensures that $\alpha-H<1/2$, and thus that one has $\int_{0}^{\infty}\frac{\min(1,r)^{1-\alpha}}{r^{\frac32-H}}\dd r<\infty$. Therefore one obtains the following inequality: for all $\alpha\in[0,2H)$, there exists $C_{\alpha,H}\in(0,\infty)$ such that for all $\lambda\in(0,\infty)$, all $\Delta t\in(0,1)$ and all $n\ge 0$ one has
\begin{equation}\label{eq:auxerror221}
\mathcal{K}_{H,\lambda,\Delta t,n}^{2,2,1}\le C_{\alpha,H} \lambda^{-2H} \bigl(\lambda\Delta t\bigr)^{2\alpha}.
\end{equation}

To deal with the term $\mathcal{K}_{H,\lambda,\Delta t,n}^{2,2,2}$, note that one has the following identity
\[
\int_{t_{\ell(s)+1}}^{t_{\ell(s)+2}}\frac{1}{(\tau-s)^{\frac32-H}}\dd \tau=\frac{1}{\frac12-H}\Bigl(\frac{1}{(t_{\ell(s)+1}-s)^{\frac12-H}}-\frac{1}{(t_{\ell(s)+2}-s)^{\frac12-H}}\Bigr),
\]
which provides two inequalities:
\[
\int_{t_{\ell(s)+1}}^{t_{\ell(s)+2}}\frac{1}{(\tau-s)^{\frac32-H}}\dd \tau\le \frac{1}{\frac12-H}\frac{1}{(t_{\ell(s)+1}-s)^{\frac12-H}},\qquad \int_{t_{\ell(s)+1}}^{t_{\ell(s)+2}}\frac{1}{(\tau-s)^{\frac32-H}}\dd \tau\le \frac{1}{\frac12-H}\frac{\Delta t}{(t_{\ell(s)+1}-s)^{\frac32-H}}.
\]
For all $\alpha\in[0,2H)$, combining the two upper bounds above by an interpolation argument, one has
\[
\int_{t_{\ell(s)+1}}^{t_{\ell(s)+2}}\frac{1}{(\tau-s)^{\frac32-H}}\dd \tau\le \frac{1}{\frac12-H}\frac{\Delta t^{\frac{\alpha}{2}}}{(t_{\ell(s)+1}-s)^{\frac12-H+\frac{\alpha}{2}}}.
\]

Therefore, one obtains
\begin{align*}
\mathcal{K}_{H,\lambda,\Delta t,n}^{2,2,2}
&\le C_H\Delta t^{\alpha}\min\bigl(1,\lambda\Delta t\bigr)^2\int_{0}^{t_n}e^{-2\lambda(t_n-s)}\frac{1}{(t_{\ell(s)+1}-s)^{1-2H+\alpha}}\dd s\\
&\le C_H\Delta t^{\alpha}\min\bigl(1,\lambda\Delta t\bigr)^2\sum_{k=0}^{n-1}e^{-2\lambda(t_n-t_{k+1})}\int_{t_k}^{t_{k+1}}\frac{1}{(t_{\ell(s)+1}-s)^{1-2H+\alpha}}\dd s\\
&\le C_{H}\Delta t^{\alpha}\min\bigl(1,\lambda\Delta t\bigr)^2\sum_{k=0}^{n-1}e^{-2\lambda(t_n-t_{k+1})}\int_{0}^{\lambda\Delta t}\frac{\lambda^{\alpha-2H}}{r^{1-2H+\alpha}}\dd r\\
&\le C_{H}\lambda^{-2H}(\lambda\Delta t)^{\alpha}\frac{\min\bigl(1,\lambda\Delta t\bigr)^2}{1-e^{-2\lambda\Delta t}} \int_{0}^{\infty}\frac{1}{r^{1-2H+\alpha}}\dd r.
\end{align*}
Assuming that $\alpha<2H$ ensures that $\int_{0}^{\infty}\frac{1}{r^{1-2H+\alpha}}\dd r<\infty$. In addition, one has
\[
\underset{z\in(0,\infty)}\sup~\frac{\min\bigl(1,z\bigr)}{1-e^{-2z}}<\infty.
\]
Therefore one obtains the following inequality: for all $\alpha\in[0,2H)$, there exists $C_{\alpha,H}\in(0,\infty)$ such that for all $\lambda\in(0,\infty)$, all $\Delta t\in(0,1)$ and all $n\ge 0$ one has
\begin{equation}\label{eq:auxerror222}
\mathcal{K}_{H,\lambda,\Delta t,n}^{2,2,2}\le C_{\alpha,H} \lambda^{-2H} \bigl(\lambda\Delta t\bigr)^{2\alpha}.
\end{equation}
Combining the inequalities~\eqref{eq:auxerror221} and~\eqref{eq:auxerror222} then yields the following inequality: for all $\alpha\in[0,2H)$, there exists $C_{\alpha,H}\in(0,\infty)$ such that for all $\lambda\in(0,\infty)$, all $\Delta t\in(0,1)$ and all $n\ge 0$ one has
\begin{equation}\label{eq:auxerror22}
\mathcal{K}_{H,\lambda,\Delta t,n}^{2,2}\le C_{\alpha,H} \lambda^{-2H} \bigl(\lambda\Delta t\bigr)^{2\alpha}.
\end{equation}
Gathering the estimates~\eqref{eq:auxerror21} and~\eqref{eq:auxerror22} then yields the following inequality: for all $\alpha\in[0,2H)$, there exists $C_{\alpha,H}\in(0,\infty)$ such that for all $\lambda\in(0,\infty)$, all $\Delta t\in(0,1)$ and all $n\ge 0$ one has
\begin{equation}\label{eq:auxerror2}
\mathcal{K}_{H,\lambda,\Delta t,n}^{2}\le C_{\alpha,H} \lambda^{-2H} \bigl(\lambda\Delta t\bigr)^{2\alpha}.
\end{equation}

Finally, gathering the inequalities~\eqref{eq:auxerror0},~\eqref{eq:auxerror1} and~\eqref{eq:auxerror2}, one obtains the inequality~\eqref{eq:error1DH-} and the proof of Lemma~\ref{lem:error1DH-} is completed.
\end{proof}

Note that combining Lemma~\ref{lem:moment1DH-} and~\ref{lem:error1DH-}, one obtains the following moment bounds for the discrete-time approximation $\bigl(\Z_n^{H,\lambda}\bigr)_{n\ge 0}$ of $\bigl(\Z^{H,\lambda}(t_n)\bigr)_{n\ge 0}$.
\begin{lemma}\label{lem:momentnum1DH-}
For all $H\in(0,1/2)$, there exists $C_H\in(0,\infty)$ such that for  all $\lambda\in(0,\infty)$ one has
\begin{equation}\label{eq:momentnum1DH-}
\underset{\Delta t\in(0,1)}\sup~\underset{n\ge 0}\sup~\E[|\Z_n^{H,\lambda}|^2]\le C_H\lambda^{-2H}.
\end{equation}
\end{lemma}
\end{appendix}


\begin{thebibliography}{}
	%------applications of fBms
	\bibitem{6} G.S. Ladde, L. Wu, Development of nonlinear stochastic models by using stock price data and basic statistics, Neutral Parallel Sci. Comput., 18(2010) 269-282.
	
	\bibitem{7} X.L. Feng, P.J. Li, X. Wang, An inverse random source problem for the time fractional diffusion equation driven by a fractional Brownian motion, Inverse Problem, 36(2020) 045008.
	
	
	\bibitem{Lobashev-2025} A. Lobashev, D. Guskov, K. Polovnikov, Generative inpainting of incomplete Euclidean distance matrices of trajectories generated by a fractional Brownian motion, Sci. Reports, 15(2025) 1-20.
	
	%----research about sdes with fBms	
	\bibitem{Duncan-2000} T.E. Duncan, Y.Z. Hu, B. Pasik-Duncan, Stochastic calculus for fractional Brownian motion I. theory, SIAM J. Control Optim., 38(2000) 582-612.
	
	\bibitem{Maslowski} B. Maslowski, D. Nualart, Evolution equations driven by a fractional Brownian motion, J. Funct. Anal., 202(2003) 277-305.
	
	\bibitem{Mishura} Y. Mishura, G. Shevchenko, The rate of convergence for Euler approximations of solutions of stochastic differential equations driven by fractional Brownian
	motion. Stochastics,
	80(2008) 489-511.
	
	\bibitem{Neuenkirch} A. Neuenkirch,
	Optimal pointwise approximation of
	stochastic differential equations
	driven by fractional Brownian motion,
	Stochastic Process Appl., 118(2008) 2294-2333.
	
	\bibitem{5} T.E. Duncan, B. Maslowski, B. Pasik-Duncan, Semilinear stochastic equations in a Hilbert space
	with a fractional Brownian motion, SIAM J. Math. Anal., 40(2009) 2286-2315.
	
	\bibitem{11} X.J. Wang, R.S. Qi, F.Z. Jiang, Sharp mean-square regularity results for SPDEs with fractional noise and optimal convergence rates for the numerical approximations, BIT Numer. Math., 57(2017) 557-585.
	
	\bibitem{10} J.L. Hong, C.Y. Huang, Super-convergence analysis on exponential integrator for stochastic heat equation driven by additive fractional Brownian motion, arXiv: 2007.02223v [math. NA] 5 Jul 2020.
	
	\bibitem{Ding-2024} X.L. Ding, D.H. Wang, Stability theory of stochastic evolution equations with multiplicative
	fractional Brownian motions in Hilbert spaces, Chaos, Solitons and Fractals, 187(2024) 115435.
	
	\bibitem{Ding-2024-1} X.L. Ding, D.H. Wang, Regularity analysis for SEEs with multiplicative fBms
	and strong convergence for a fully discrete scheme, IMA J. Numer. Anal., 44(2024) 1435-1463.
	
	%---sdes driven by fBms with non-L conditions
	
	\bibitem{Liu-2022} Z.H. Liu, $L^p$-convergence rate of backward Euler schemes
	for monotone SDEs, BIT Numer. Math., 62(2022) 1573-1590.
	
	\bibitem{Zhou-2023} H. Zhou, Y.Z Hu, Y.H. Liu, Backward Euler method for stochastic differential
	equations with non-Lipschitz coefficients driven by
	fractional Brownian motion, BIT Numer. Math., (2023) 63:40.
	
	%-----stochastic Allen-Cahn equations
	
	\bibitem{Allen-1979} S. Allen, J. Cahn, A macroscopic theory for antiphase boundary motion and its application to
	antiphase domain coarsening, Acta Metal. Mater., 27(1979) 1085-1095.
	
	\bibitem{Kohn-2007} R. Kohn, F. Otto, M. G. Reznikoff, E. Vanden-Eijnden,  Action minimization and sharp-interface
	limits for the stochastic Allen-Cahn equation, Comm. Pure Appl. Math., 60(2007) 393-438.
	
	\bibitem{Lindgren-2015} M. Kov\'{a}cs, S. Larsson, F. Lindgren, On the backward Euler approximation of the stochastic Allen-Cahn equation, J. Appl. Probab., 52(2015) 323-338.
	
	%%------stochastic Allen-Cahn equations
	\bibitem{Becker-2017-1} S. Becker, B. Gess, A. Jentzen, P.E. Kloeden, Strong convergence rates for explicit space-time
	discrete numerical approximations of stochastic Allen-Cahn equations, Stochastics Partial Diff. Equat.: Anal. Comput., 11(2023) 211-268.
	
	\bibitem{Lindgren-2015-1} M. Kov\'{a}cs, S. Larsson, F. Lindgren. On the discretisation in time of the stochastic Allen-Cahn
	equation, Math. Nachrichten, 291(2018) 966-995. 
	
	\bibitem{Liu-2017-1} Z. Liu, Z. Qiao, Wong-Zakai approximations of stochastic Allen-Cahn equation, Int. J. Numer. Anal. Model., 16(2019) 681-694. 
	
	\bibitem{Brehier-18-01} C.E. Br\'{e}hier, L. Gouden\`{e}ge, Analysis of some splitting schemes for the stochastic
	Allen-Cahn equation, Discrete Contin. Dyn. Syst. Ser. B, 24(2019) 4169-4190.  
	
	\bibitem{Brehier-19-02} C.E. Br\'{e}hier, J.B. Cui, J.L. Hong, Strong convergence rates of semidiscrete splitting
	approximations for the stochastic Allen-Cahn equation, IMA J. Numer. Anal., 39(2019) 2096-2134. 
	
	\bibitem{Brehier-20-02} C.E. Br\'{e}hier, L. Gouden\`{e}ge, Weak convergence rates of splitting schemes for the
	stochastic Allen-Cahn equation,  BIT Numer. Math., 60(2020) 543-582.
	
	\bibitem{Higham-2002} D.J. Higham, X.R Mao, A.M. Stuart, Strong convergence of Euler-type methods for
	nonlinear stochastic differential equations, SIAM J. Numer. Anal., 40(3)(2002) 1041-1063. 
	
	\bibitem{Zamani-2010} S. Zamani, Reaction-diffusion equations with polynomial drifts driven by fractional Brownian motions, Stochastic Anal. Appl., 28(2010) 1020-1039.
	
	\bibitem{Becker-2018} S. Becker, A. Jentzen, Strong convergence rates for nonlinearity-truncated Euler-type
	approximations of stochastic Ginzburg-Landau equations, Stochastic Processes Appl., 129(1)(2019) 28-69.
	
	\bibitem{Qi-2019} R.S. Qi, X.J. Wang, Optimal error estimates of Galerkin finite element methods for stochastic Allen-Cahn
	equation with additive noise, J. Sci. Comput., 80(2)(2019) 1171-1194.
	
	%-------- tool books
	%%%------------with fractional Brownian motions
	
	\bibitem{Yosida} K. Yosida, Functional analysis, 6th edition, Springer-Verlag, Berlin, 1980.
	
	\bibitem{Pruss} J. Pr\"{u}ss, On the spectrum of $C_0$-semigroups, Trans. Amer. Math. Soc., 284(1984) 847-857.
	
	\bibitem{Prato} G. Da Prato, J. Zabczyk, Stochastic equations in infinite dimensions, Second Edition, Cambridge University Press, 2014.
	
	\bibitem{4} F. Biagini, Y.Z. Hu, B. {\O}ksendal, T.S. Zhang, Stochastic calculus
	for fractional Brownian motion and applications, Springer, 2008.
	
	\bibitem{Hu-2005} Y.Z. Hu, Integral transformations and anticipative calculus for fractional Brownian motions, Methods of the American Mathematical Society, 2005.
	
	\bibitem{Nualart-2006} D. Nualart. The Malliavin calculus and related topics. 2nd Edition, Springer-Verlag, Heidel-
	berg, 2006.
	
\end{thebibliography}
\end{document}